\documentclass{book}
\usepackage{amsopn,amstext,amsbsy,amsmath,amscd,amsthm,amsfonts}
\usepackage[ma,mnf,master]{mccover} 
\usepackage{geometry} 
\usepackage[latin9]{inputenc} 
\usepackage[all]{xy} 

\usepackage{hyperref}

\usepackage{amssymb}
\usepackage[toc,page]{appendix}

\usepackage{amsmath}
\usepackage[makeroom]{cancel}
\usepackage{array}

\newcolumntype{L}[1]{>{\raggedright\arraybackslash}p{#1}}
 
\newcolumntype{C}[1]{>{\centering\arraybackslash}p{#1}}
 
\newcolumntype{R}[1]{>{\raggedleft\arraybackslash}p{#1}}

\theoremstyle{plain}
\newtheorem{theorem}                 {\bf Theorem}      [chapter]

\newtheorem{lemma}        [theorem]  {Lemma}
\newtheorem{proposition}  [theorem]  {Proposition}

\theoremstyle{definition}

\newtheorem{definition}   [theorem]  {Definition}

\newtheorem{example}      [theorem]  {Example}

\newtheorem{remark}       [theorem]  {Remark}

\numberwithin{equation}{chapter}

\allowdisplaybreaks

\def \rn{{\mathbb R}}

\def \A{\mathcal A}

\def \F{\mathcal F}

\def \H{\mathcal H}
\def \I{\mathcal I}

\def \V{\mathcal V}

\def\nab#1#2{\hbox{$\nabla$\kern -.3em\lower 1.0 ex
		\hbox{$#1$}\kern -.1 em {$#2$}}}

\def \lb#1#2{[#1,#2]}

\def \g{\mathfrak{g}}

\def \k{\mathfrak{k}}

\def \m{\mathfrak{m}}

\def \SL2{\widetilde{\text{\bf SL}}_{2}(\rn)}

\def \SU#1{\text{\bf SU}(#1)}

\DeclareMathOperator{\Div}{div}

\DeclareMathOperator{\grad}{grad}

\DeclareMathOperator{\trace}{trace}

\newcommand{\R}{\mathbb{R}}
\newcommand{\C}{\mathbb{C}}

\newcommand{\K}{\mathcal{K}}

\newcommand{\id}{\operatorname{id}}

\allowdisplaybreaks

\maintitle{Natural Almost Hermitian Structures \\ on Conformally Foliated 4-Dimensional \\ Lie
Groups with Minimal Leaves}
\authors{Emma Andersdotter Svensson}
\issnum{2022}{5}
\sernum{LUNFMA}{3128}{2022}

\newif\ifDemo
\Demofalse

\newif\ifNotesForLater
\NotesForLaterfalse

\begin{document}

\def\nab#1#2#3{\nabla^{\hbox{$\scriptstyle{#1}$}}_{\hbox{$\scriptstyle{#2}$}}{\hbox{$#3$}}}

\def\snab#1#2#3{\hbox{$\nabla$\kern-.1em\raise 1.2 ex\hbox{$\scriptstyle{#1}$}\kern-.5em\lower 0.8 ex\hbox{$#2$}\kern-.0em{$#3$}}}

\def\nnab#1{\nabla^{\hbox{$\scriptstyle{#1}$}}}
\def\nsnab#1{\hbox{$\nabla$\kern-.1em\raise 1.0 ex\hbox{$\scriptstyle{#1}$}}}

\def\rn{\mathbb R}
\def\ci{\mathcal I}
\def\g{\mathfrak{g}}
\def\tr{\textrm{\upshape{trace}}}
\def\pdt{\frac{\partial}{\partial t}}


\large 

\frontcover 

\thispagestyle{empty}
\centerline {\bf\Large Abstract}
\vskip1cm
Let $(G,g)$ be a 4-dimensional Riemannian Lie group with a  2-dimensional left-invariant, conformal foliation $\F$ with minimal leaves. Let $J$ be an almost Hermitian structure on $G$ adapted to the foliation $\F$.  The corresponding Lie algebra $\g$ must then belong to one of 20 families $\g_1,\dots,\g_{20}$ according to \cite{Gud-Sve-6}. We classify such structures $J$ which are almost K\"{a}hler $(\A\K)$, integrable $(\I)$ or K\"{a}hler $(\K)$. Hereby, we construct 16 multi-dimensional almost K\"{a}hler families, 18 integrable families and 11 K\"{a}hler families.

\vskip13cm
{\it Throughout this work it has been my firm intention to give reference to the stated results and credit to the work of others. All theorems, propositions, lemmas and examples left unmarked are either assumed to be too well known for a reference to be given or the fruits of my own efforts.}


\newpage 
\thispagestyle{empty}
\phantom{m}

\newpage 
\thispagestyle{empty}
\centerline {\bf\Large Acknowledgments}
\vskip1cm
I would like to thank my supervisor Sigmundur Gudmundsson for helping me find the project and guiding me throughout the process.
\vskip1pc
\hskip9cm Emma Andersdotter Svensson
\phantom{m}

\newpage 
\thispagestyle{empty}
\phantom{m}
\newpage
\thispagestyle{empty}
\centerline {\bf\Large Popular Scientific Summary}
\vskip1cm
In any scientific field, it is important to classify certain objects in order to form a greater understanding. Here, this is done for so-called 4-dimensional Lie groups with some additional specifications. The theory of Lie algebras and Lie groups is connected to almost all topics in mathematics and is also widely used in physics. For instance, gauge theory in physics uses certain symmetries of a physical system to form a Lie group. It is through the use of gauge theory that the existence of certain particles has been postulated, which points towards a greater understanding of our universe.

Throughout this work, we use a field in mathematics called Riemannian geometry. Euclidean geometry is what most people have been using when working with elementary geometry. 
Here, lines are assumed to be straight in the sense that two parallel lines will never meet, no matter how long you make them. It is, however, sometimes useful to work with geometry where lines can be assumed to not be straight. According to general relativity, for example, the path of a light particle bends when it is in the presence of a massive body. Such geometry is referred to as non-Euclidean geometry. Riemannian geometry was developed in the 19th century by Bernhard Riemann. In 1854, he presented his ideas in a lecture hall at his university in G\"{o}ttingen. Among the audience was his former teacher Carl Friedrich Gauss. With his new theory, Riemann provided a system to unite all non-Euclidean geometries.

Riemannian geometry makes it possible to define a so-called manifold. A manifold is a space that up-close looks Euclidean. For instance, if a person stands on a sphere as big as the Earth and looks around, they appear to be standing on a flat plane. A Lie group is defined as a manifold that also has the properties of a group. A mathematical group is an arbitrary set paired with an operation (for example addition or multiplication) that is associative, has an identity element and any element must have an inverse.
For example, the set of whole numbers paired with addition is a group. Any Lie group has an associated Lie algebra, which is the associated Euclidean space at a point of the Lie group. In the case of a person standing on a sphere of the size of the Earth, imagine that the person makes the perceived plane at the point at which they are standing infinitely larger.

The 4-dimensional Lie group studied in this thesis is also a so-called almost Hermitian manifold, which can be seen as a generalisation of a complex space. With the specifications that will be given to this Lie group, it has been shown by S. Gudmundsson and M. Svensson that its corresponding Lie algebras can be divided into 20 families. We use the work of A. Gray and L. Hervella to show that any 4-dimensional almost Hermitian manifold must belong to one of four classes and determine when each of the 20 families belongs to the four classes.

\newpage 
\thispagestyle{empty}

\newpage 
\tableofcontents
\thispagestyle{empty}
\phantom{m}

\newpage 
\setcounter{page}{0} 
\thispagestyle{empty}

\chapter{Introduction}
The reader is assumed to have a fundamental knowledge of Riemannian geometry corresponding to that of \cite{Gud-Riemann}.

Let $(G,g)$ be a 4-dimensional Lie group with a 2-dimensional foliation $\F$ which is minimal, conformal and left-invariant. Such a group carries a natural almost complex structure $J$. Denote by $\g$ the corresponding Lie algebra of $G$. It was proven by S. Gudmundsson and M. Svensson that $\g$ must belong to one of 20 families $\g_1,\dots ,\g_{20}$, see \cite{Gud-Sve-6}. The aim of this Master's thesis is to determine when the structure $J$ is integrable ($\mathcal{I}$), almost K\"{a}hler ($\mathcal{AK}$) or K\"{a}hler ($\mathcal{K}$). 
The results are new and summarized in Table \ref{tab:J1}.

In Chapter \ref{background}, we introduce some useful theory about so-called harmonic morphisms and foliations.

In Chapter \ref{4dimliegr}, we introduce 4-dimensional Lie groups $(G,g)$ which are equipped with a minimal, conformal and left-invariant foliation $\F$ of dimension 2. These Lie groups correspond to those in the article \cite{Gud-Sve-6}. We show that the corresponding Lie algebras $\g$ of $(G,g)$ must have the Lie bracket relations given in Equation \eqref{LieBraRela}, as stated in the original article \cite{Gud-Sve-6}.

The theory of almost Hermitian manifolds is introduced in Chapter \ref{AlmostHermitianManifolds}, where we go over concepts like (almost) complex structures, Hermitian manifolds and K\"{a}hler manifolds, which will be useful later on.

Chapter \ref{SpaceW} is based on \cite{Gra-Her} and introduces the space $W$ consisting of covariant derivatives of the K\"{a}hler form. We show that an almost complex structure $J$ on a 4-dimensional almost Hermitian manifold must be either almost Hermitian ($\mathcal{W}$), integrable $(\I)$, almost K\"{a}hler $(\mathcal{AK})$ or K\"{a}hler $(\K)$.

In Chapter \ref{W2W4} we apply this to the 4-dimensional Lie groups introduced in Chapter \ref{4dimliegr} and determine when they are almost K\"{a}hler, integrable or K\"{a}hler. 

The main work takes place in Chapter \ref{23families}, where we look at the 20 families of Lie algebras found in \cite{Gud-Sve-6} and construct new examples in each family that are either almost K\"{a}hler $(\mathcal{AK})$, integrable $(\I)$ or K\"{a}hler $(\K)$.

\chapter{Harmonic Morphisms and Foliations}\label{background}
In this chapter, we go over some useful theory of harmonic morphisms and foliations. We follow \cite{Fug}.

\section{Harmonic Morphisms}
In 1848, Jacobi published a paper where he introduced the idea of harmonic morphisms in the 3-dimensional Euclidean space \cite{Jac}. A harmonic morphism is, broadly speaking, a map preserving Laplace's equation. His ideas were then generalized to maps between Riemannian manifolds. In this section, we introduce harmonic morphisms in the case of Riemannian manifolds. We will, occasionally, draw parallels to the work done by Jacobi.

Given a two times differentiable function $f$ defined on an open subset of the Euclidean space $\R^m$, the \textit{Laplacian} $\Delta$ of $f$ is given by
$$
\Delta f=\frac{\partial^2f}{\partial x_1^2}+\dots+\frac{\partial^2f}{\partial x_m^2}.
$$
Notice that $\Delta f=\Div\grad f$. Laplace's equation in $\R^m$ is given by $\Delta f=0$. The solutions of Laplace's equation are called \textit{harmonic functions}. In Definition \ref{LaplacianDfn} we define harmonic functions on a general Riemannian manifold, similarly to how it is done on $\R^m$. To do this, we must first introduce some additional definitions.

\begin{definition}\label{DivGrad}
Let $f$ be a smooth real-valued function on the smooth Riemannian manifold $(M,g)$. Then the \textit{gradient} of $f$ is a vector field given by
\begin{equation*}
    g(\grad{f},Y)=\mathrm{d}f(Y),
\end{equation*}
where $\mathrm{d}f(Y)$ is the differential of $f$ and $Y$ is any vector field on $M$.

Let $E$ be an arbitrary vector field. Then the \textit{divergence} of $E$ is given by
\begin{equation*}
    \Div{E}=\trace{\nabla E},
\end{equation*}
where $\nabla$ is the Levi-Civita connection on $(M,g)$.
\end{definition}
We are now ready to introduce the Laplace-Beltrami operator and define what is meant by a harmonic function on a Riemannian manifold.
\begin{definition}\label{LaplacianDfn}
Let $(M,g)$ be a Riemannian manifold. The \textit{Laplace-Beltrami operator} $\tau$ of any function $f:U\subset M\to\R$ of class $C^2$  is defined by
\begin{equation*}
    \tau(f)=\Div\grad{f}
    =\Div\mathrm{d}f,
\end{equation*}
where the divergence and gradient are given in Definition \ref{DivGrad}. \textit{Laplace's equation} is given by $\tau(f)=0$ and its solutions are called \textit{harmonic functions}. The operator $\tau$ is also called the \textit{tension field}.
\end{definition}

The tension field for a smooth map $\phi:(M,g)\to(N,h)$ is similarly defined by
\begin{equation*}
    \tau(\phi)
    =\Div\mathrm{d}\phi,
\end{equation*}
where $\mathrm{d}\phi$ is the differential of $\phi$. The map $\phi$ is harmonic if and only if $\tau(\phi)=0$, see \cite[Theorem 3.3.3]{Fug}.

We now define what is meant by a harmonic morphism between Riemannian manifolds.
\begin{definition}
Let $(M,g)$ and $(N,h)$ be two Riemannian manifolds. A smooth mapping $\phi:M\rightarrow N$ is called a \textit{harmonic morphism} if, for any harmonic function $f:U\rightarrow\R$ on an open subset $U$ of $N$ such that $\phi^{-1}(U)$ is non-empty, the composition $f\circ \phi$ is harmonic on $\phi^{-1}(U)$.
\end{definition}

A complex-valued function $\phi$ on $\R^3$ given by $\phi=u+i\,v$, where $u$ and $v$ are real-valued functions on $\R^3$, is called \textit{horizontally (weakly) conformal} if and only if
\begin{equation}\label{HorConInC}
|\nabla u|=|\nabla v|\ \ \textrm{and} \ \ \langle\nabla u, \nabla v\rangle=0.    \end{equation}
Jacobi showed in his article from 1848 that a complex-valued function on $\R^3$ of class $C^2$ is a harmonic morphism if and only if it is harmonic and horizontally conformal. The result was independently generalized by Fuglede in 1978 \cite{Fug} and by Ishihara in 1979 \cite{Ishihara} to Riemannian manifolds. We present this in Theorem \ref{FugIsh} below.

We now give an example of how the theorem by Jacobi shows that the usual holomorphic and anti-holomorphic functions in complex analysis are harmonic morphisms.
\begin{example}
Let $f:U\subset \C\rightarrow\C$ be a function of class $C^2$ such that $f=u+i\,v$, where $u,v:U\rightarrow\R$. Then $f$ is  horizontally weakly conformal if and only if
$$\langle\nabla{u},\nabla{v}\rangle=0 \ \ \textrm{and} \ \ |\nabla{u}|^2=|\nabla{v}|^2.$$
That is,
\begin{equation*}
    0=\langle\nabla{u},\nabla{v}\rangle
    =
    \frac{\partial u}{\partial x}\frac{\partial v}{\partial x}
    +
    \frac{\partial u}{\partial y}\frac{\partial v}{\partial y}
\end{equation*}
and
\begin{equation*}
    0=|\nabla{u}|^2-|\nabla{v}|^2=\left(\frac{\partial u}{\partial x}\right)^2+\left(\frac{\partial u}{\partial y}\right)^2
    -
    \left(\frac{\partial v}{\partial x}\right)^2-\left(\frac{\partial v}{\partial y}\right)^2.
\end{equation*}
This is true if and only if the function $f$ or its conjugate $\Bar{f}$ satisfy the classical \textit{Cauchy-Riemann equations}:
$$\left(\frac{\partial u}{\partial x},\frac{\partial u}{\partial y}
\right)=\pm\left(\frac{\partial v}{\partial y},-\frac{\partial v}{\partial x}\right).
$$
Thus $f$ is also \textit{holomorphic} or \textit{anti-holomorphic}. Notice that, if this is true, then
\begin{equation*}
    \Delta u=\frac{\partial^2 u}{\partial x^2}+\frac{\partial^2 u}{\partial y^2}
    =\pm
    \left(\frac{\partial}{\partial x}\frac{\partial v}{\partial y}-\frac{\partial}{\partial y}\frac{\partial v}{\partial x}\right)=0
\end{equation*}
and
\begin{equation*}
    \Delta v=\frac{\partial^2 v}{\partial x^2}+\frac{\partial^2 v}{\partial y^2}=
    \pm\left(-\frac{\partial}{\partial x}\frac{\partial u}{\partial y}+\frac{\partial}{\partial y}\frac{\partial u}{\partial x}\right)=0,
\end{equation*}
giving $\Delta f=\Delta u+i\Delta v=0$. Thus $f$ is also \textit{harmonic}. It follows that the function $f$ is a harmonic morphism.
\end{example}

Let us remind the reader what is meant by a (weakly) conformal map between Riemannian manifolds.

\begin{definition}
Given two Riemannian manifolds $(M,g)$ and $(N,h)$, a smooth map $\phi:M\to N$ is said to be (weakly) conformal at a point $p\in M$ if there exists a real non-negative number $\lambda(p)$ so that
\begin{equation*}
    h(\mathrm{d}\phi_p(E),\mathrm{d}\phi_p(F))=\lambda^2(p)\,g(E,F)
\end{equation*}
for each $E,F\in T_pM$. The number $\lambda(p)$ is called the \textit{conformality factor}. If the above holds for all points $p\in M$, the map $\phi$ is called \textit{(weakly) conformal}.
\end{definition}

We will now define what is meant by a horizontally weakly conformal map in the case of Riemannian manifolds. Let $\phi:(M^m,g)\to(N^n,h)$ be a smooth map between Riemannian manifolds and let $p\in M$ be any point in $M$. Then the \textit{vertical space} $\V_p$ of $\phi$ at $p$ is the kernel of $\mathrm{d}\phi_p$ and the \textit{horizontal space} $\H_p$ of $\phi$ at $p$ is the orthogonal complement of $\V_p$.

\begin{definition}\label{HorConDfn}
Let $\phi:(M^m,g)\to(N^n,h)$ be a smooth map between Riemannian manifolds. The map $\phi$ is then said to be \textit{horizontally weakly conformal at a point $p\in M$} if
\begin{itemize}
    \item[(i)] $\mathrm{d}\phi_p=0$, or
    \item[(ii)] the differential $\mathrm{d}\phi_p$ is surjective and there exists a positive real number $\lambda(p)$ such that
    $$h(\mathrm{d}\phi_p(X),\mathrm{d}\phi_p(Y))=\lambda^2(p)\,g(X,Y),
    $$
    for any $X,Y\in\H_p$.
\end{itemize}
The map $\phi$ is called \textit{horizontally weakly conformal} if the above holds for all points in $M$.
\end{definition}

\begin{definition}\label{DfnSubmersion}
For a smooth map $\phi:(M^m,g)\to (N^n,h)$ between Riemannian manifolds, the points $p\in M$ at which $\mathrm{d}\phi_p=0$ are called \textit{critical points}. The points $q\in M$ where the restriction $\mathrm{d}\phi_q|_{\H_q}$ is surjective and conformal are called \textit{regular points}. $\phi$ is called a \textit{conformal submersion} if it has no critical points and is horizontally weaky conformal at every point in $M$.
\end{definition}

\begin{remark}
Definition \ref{HorConDfn} can be written more compactly if, given a smooth map $\phi:(M^m,g)\to(N^n,h)$ between Riemannian manifolds, we define a set $C_\phi$ by
$$C_\phi=\{p\in M\,|\, \mathrm{d}\phi_p=0\}$$
called the \textit{critical set}. The map $\phi$ is then horizontally weakly conformal if for each $p\in M\setminus C_\phi$, the restriction $\mathrm{d}\phi_p|_{\mathcal{H}_p}$ is surjective and conformal.
\end{remark}

The following result was proved by Fuglede in 1978 \cite{Fug} and by Ishihara in 1979 \cite{Ishihara}, independently of each other.

\begin{theorem}[\cite{Fug}, \cite{Ishihara}]\label{FugIsh}
A map $\phi:(M,g)\rightarrow(N,h)$ of class $C^2$ between Riemannian manifolds is a harmonic morphism if and only if it is harmonic and horizontally weakly conformal.
\end{theorem}

\section{Foliations}
\begin{definition}\label{DfnFoliation}
For $k\in\{0,1,\dots,\infty,\omega\}$, let $M$ be an $m$-dimensional manifold of class $C^k$ and let
$n,q$ be positive integers such that $m=n+q$. 
 A \textit{foliation} $\F$ of $M$, of dimension $q$, is a decomposition $\F=\{L_\alpha\}$ of $M$ into disjoint, connected submanifolds, each of dimension $q$, such that:
Given a point $p$ in $M$, there is a submersion $\phi_p:U_p\rightarrow N^n$ of class $C^k$ from an open neighbourhood $U_p$ of $M$ to an $n$-dimensional manifold $N^n$. Further, the connected components of $U_p\cap L_\alpha$ are the fibres of $\phi_p$, i.e. for each $\alpha$ there exists a point $y_\alpha\in N^n$ such that
$$U_p\cap L_\alpha=\{x\in U_p\,|\,\phi_p(x)=y_\alpha\}.$$
The manifolds $L_\alpha$ are called the \textit{leaves} of the foliation. 

The connected components of the fibres of a smooth submersion $\phi$ form the leaves of a foliation. Such a foliation is called the \textit{foliation associated with $\phi$}.
\end{definition}

\begin{remark}
Note that we can, equivalently, define a foliation in terms of local coordinates. Let $M$ be an $m$-dimensional Riemannian manifold with local coordinates $(x^1,\dots,x^m)$ on an open neighbourhood $U$ of a point $p$ in $M$. Then each connected component $U\cap L_\alpha$, where $L_\alpha$ is given by Definition \ref{DfnFoliation}, is given by $(x^{q+1},\dots,x^{m})=\textrm{constant}$.
\end{remark}

Given a vector bundle $E\stackrel{\pi}{\to}M$, we denote by $C^\infty(E)$ the set of all its smooth sections.

\begin{definition}\label{InvDfn}
Let $M$ be a manifold of dimension $m$ and let $k$ be a positive integer less than or equal to $m$. Then a $k$-dimensional \textit{distribution} $\V $ is a $k$-dimensional subbundle of the tangent bundle $TM$.

Denote by $\lb{\cdot}{\cdot}$ the Lie bracket. A distribution $\V$ is called \textit{involutive} (or \textit{integrable}) if $\lb{V}{W}\in C^\infty(\V)$ for all $V,W\in C^\infty(\V)$, i.e. $\V$ is closed under the Lie bracket.
\end{definition}
The following theorem is called \textit{Frobenius' theorem}.
\begin{theorem}[\cite{Fug}]\label{FrubenThrm}
Suppose that $M$ is an $m$-dimensional manifold with a $q$-dimensional smooth involutive distribution $\V$. Let $N$ be any $q$-dimensional submanifold of $M$ such that $T_pN=\V_p$ for all points $p$ in $N$. Then the connected components of $N$ form the leaves of a smooth $q$-dimensional foliation $\F$. The tangent spaces of $\F$ are given by $\V$.

The converse is also true: Let $\F$ be a smooth foliation. Then its tangent spaces form an involutive distribution.
\end{theorem}
Now let $(M^m,g)$ be a Riemannian manifold. Denote by $\mathcal{V}$ a $q$-dimensional distribution on $M$ and $\mathcal{H}$ its orthogonal complement distribution on $M$. Then $$TM=\V\oplus\H.$$
We call the distributions $\V$ and $\H$ the \textit{vertical} and \textit{horizontal distribution}, respectively.

The second fundamental form of $\mathcal{V}$ and $\mathcal{H}$ are tensor fields of type $(2,1)$ given by
$$B^\mathcal{V}(U,V)=\frac {1}{2}\mathcal{H}(\nabla_UV+\nabla_VU)\qquad(U,V\in C^\infty(\mathcal{V}))$$
and
$$B^\mathcal{H}(X,Y)=\frac{1}{2}\mathcal{V}(\nabla_XY+\nabla_YX)\qquad(X,Y\in C^\infty(\mathcal{H})),$$
respectively.
Given the second fundamental form $B^\V$, we can define the \textit{mean curvature} of $\V$:
\begin{equation*}
    \mu^\V=\frac{1}{q}\trace{B^\V}.
\end{equation*}

We can now define what it means for the vertical distribution $\V$ to be \textit{minimal}, \textit{totally geodesic}, \textit{conformal} or \textit{Riemannian}:
\begin{definition}\label{PropOfVDfn}
Let $\V$ be the vertical distribution of a Riemannian manifold $(M,g)$ as above. Then $\V$ is
\begin{itemize}
    \item[(i)] \textit{minimal} if the mean curvature $\mu^\V$ vanishes at every point in $M$;
    \item[(ii)] \textit{totally geodesic} if $B^\V_p=0$ for any point $p$ in $M$;
    \item[(iii)] \textit{conformal} if there exists a vector field $V$ in $\mathcal{V}$ such that
\begin{equation*}
B^\mathcal{H}=g\otimes V;\end{equation*}
    \item[(iv)] \textit{Riemannian} if, in addition to being conformal, the vector field $V$ in $\mathcal{V}$ is identically zero.
\end{itemize}
\end{definition}
We end this section by stating a theorem connecting harmonic morphisms to foliations of codimension 2.
\begin{theorem}[{\cite[Proposition 4.7.1]{{Fug}}}]\label{FoliAndHarm}
Let $\F$ be a conformal foliation of codimension 2 on a Riemannian manifold $M$. Then the following is true if and only if $\F$ has minimal leaves: For any point $p$ in $M$, there exists an open neighbourhood $U$ such that the restriction $\F|_U$ is associated to a submersive harmonic morphism $\phi:U\to N^2$, where $N^2$ is a 2-dimensional Riemannian manifold.
\end{theorem}

\chapter{4-Dimensional Lie groups}\label{4dimliegr}
In this chapter, we apply some of the theory from Chapter \ref{background} as well as our previous knowledge to construct the same $4$-dimensional Lie algebra as the one given in \cite{Gud-Sve-6}. This is done by introducing a $4$-dimensional Lie group $(G,g)$ equipped with a left-invariant, minimal and conformal foliation of codimension 2. By Theorem \ref{FoliAndHarm} such a foliation is locally given by submersive harmonic morphism. Our main goal of this chapter is to show that this Lie algebra will have the Lie bracket relations \eqref{LieBraRela}, as stated in \cite{Gud-Sve-6}.

\section{Setting the Stage}
Let $(G,g)$ be a Riemannian manifold. We denote by $\V$ an involutive distribution on $G$ and by $\H$ the orthogonal complementary distribution of $\V$. $\V$ and $\H$ also denote the orthogonal projections onto the corresponding subbundles of $TG$. Then the second fundamental form of $\mathcal{V}$ is
$$B^\mathcal{V}(U,V)=\frac {1}{2}\mathcal{H}(\nabla_UV+\nabla_VU)\qquad(U,V\in C^\infty(\mathcal{V})),$$
and the second fundamental form for $\mathcal{H}$ is given by
$$B^\mathcal{H}(X,Y)=\frac{1}{2}\mathcal{V}(\nabla_XY+\nabla_YX)\qquad(X,Y\in C^\infty(\mathcal{H})).$$

By Theorem \ref{FrubenThrm}, the distribution $\V$ has an associated $2$-dimensional foliation $\F$. $\F$ is said to be \textit{conformal} if there is a vector field $V$ in $\mathcal{V}$ such that
\begin{equation}\label{confFoli}
B^\mathcal{H}=g\otimes V.\end{equation} A conformal foliation is called \textit{Riemannian} if $V=0$. $\F$ is called \textit{minimal} if $\trace{B^\mathcal{V}}=0$ and \textit{totally geodesic} if $B^\mathcal{V}=0$. This is equivalent to saying that the leaves of $\F$ are minimal and totally geodesic, respectively.

We now specify the Riemannian manifold $(G,g)$ further, by letting $G$ be a $4$-dimensional Lie group and $g$ be a left-invariant Riemannian metric. Recall that a Lie group $G$ has an associated Lie algebra $\g$ of the same dimension as $G$ given by the set of all left-invariant vector fields of $G$. 
Let $K$ be a $2$-dimensional subgroup of $G$ and denote by $\k$ the Lie algebra of $K$. Let $\m$ be the orthogonal complement of $\k$:
$$
\m=\{X\in\g\,|\,g(X,Y)=0\ \textrm{for all} \ Y\in\k\}.
$$
We then let the Lie algebra $\k$ generate the involutive distribution $\V$ and $\m$ its orthogonal distribution $\H$. Let $\{X,Y,Z,W\}$ be an orthonormal basis of $\g$ such that $Z,W$ generate $\V$ and $X,Y$ generate $\H$. The foliation $\F$ is assumed to be minimal and conformal. Denote by $\lb{\cdot}{\cdot}$ the Lie brackets on $\g$. In the article \cite{Gud-Sve-6}, they state that the Lie bracket relations of $\g$ are of the form
\begin{equation}\label{LieBraRela}
  \begin{aligned}
\lb WZ\ =&\ \ \lambda W,\\
\lb ZX\ =&\ \ \alpha X +\beta Y+z_1 Z+w_1 W,\\
\lb ZY\ =&\ \ -\beta X+\alpha Y+z_2 Z+w_2 W,\\
\lb WX\ =&\ \      a X     +b Y+z_3 Z-z_1W,\\
\lb WY\ =&\ \     -b X     +a Y+z_4 Z-z_2W,\\
\lb YX\ =&\ \      r X         +\theta_1 Z+\theta_2 W,
  \end{aligned}
\end{equation}
for real constant. We are now going to show that this is true for the remainder of this chapter. It is important to note that the system \eqref{LieBraRela} alone does not necessarily describe a Lie algebra, since a Lie algebra must also satisfy the Jacobi identity.

The elements $W,Z$ in $\g$ can be picked so that $\lb WZ\ =\ \ \lambda\, W$ since the field $\V$ is involutive and thus closed under the Lie bracket, see Definition \ref{InvDfn}. The elements $X,Y$ are picked so that $\H\lb{Y}{X}=r\,X$ for some real structure constant $r$. Note that $\H$ is involutive if and only if $\theta_1=\theta_2=0$.

For the remaining Lie brackets, we use the fact that the Lie brackets of the base vectors are linear combinations of the base vectors, meaning that we get
\begin{eqnarray}
\left[Z,X\right]&=&\alpha X+\beta Y+z_1 Z+w_1 W, \label{ZX}\\
\left[Z,Y\right]&=&a_2X+b_2Y+z_2 Z+w_2 W,\label{ZY}\\
\left[W,X\right]&=&a X+b Y+z_3 Z+d_3W,\label{WX} \\
\left[W,Y\right]&=&a_4X+b_4Y+z_4 Z+d_4W,\label{WY}
\end{eqnarray}
for some real structure constants. We are now going to use minimality and confirmality of the foliation $\F$ to obtain the same Lie bracket relations as in Equation \eqref{LieBraRela}.

\section{Minimality}
The foliation $\mathcal{F}$ is minimal in $G$ if and only if $\mathrm{trace }(B^\mathcal{V})=0$. By \cite[Definition 6.23]{Gud-Riemann} the trace is given by
\begin{eqnarray*}
    \mathrm{trace }(B^\mathcal{V})
    &=&\frac{1}{2}\mathcal{H}(\nabla_ZZ+\nabla_ZZ)
    +\frac{1}{2}\mathcal{H}(\nabla_WW+\nabla_WW)\\
    &=&\mathcal{H}(\nabla_ZZ+\nabla_WW).
\end{eqnarray*}
Thus $0=\mathcal{H}(\nabla_ZZ+\nabla_WW)$. Using this, the Koszul formula for Riemannian Lie groups \cite[Proposition 6.13]{Gud-Riemann} and the fact that the elements $X$ and $Y$ are orthogonal to the vertical distribution $\mathcal{V}$, we can compute
\begin{eqnarray*}
    0
    &=&
    g(\mathcal{H}(\nabla_ZZ+\nabla_WW),X)
    =
    g(\nabla_ZZ+\nabla_WW,X)\\
    &=&
    g(\nabla_ZZ,X)+g(\nabla_WW,X)\\
    &=&
    \frac{1}{2}\{g(X,\left[Z,Z\right])+
    g(\left[X,Z\right],Z)+
    g(Z,\left[X,Z\right])\}\\
    &&+
    \frac{1}{2}\{g(X,\left[W,W\right])+
    g(\left[X,W\right],W)+
    g(W,\left[X,W\right])\}\\
    &=&g(\left[X,Z\right],Z)+g(\left[X,W\right],W)=
    -g(\left[Z,X\right],Z)-g(\left[W,X\right],W)\\
    &=&-z_1 -d_3
\end{eqnarray*}
and
\begin{eqnarray*}
    0
    &=&
    g(\mathcal{H}(\nabla_ZZ+\nabla_WW),Y)
    =
    g(\nabla_ZZ+\nabla_WW,Y)\\
    &=&
    g(\nabla_ZZ,Y)+g(\nabla_WW,Y)\\
    &=&
    \frac{1}{2}\{g(Y,\left[Z,Z\right])+
    g(\left[Y,Z\right],Z)+
    g(Z,\left[Y,Z\right])\}\\
    &&+
    \frac{1}{2}\{g(Y,\left[W,W\right])+
    g(\left[Y,W\right],W)+
    g(W,\left[Y,W\right])\}\\
    &=&g(\left[Y,Z\right],Z)+g(\left[Y,W\right],W)=
    -g(\left[Z,Y\right],Z)-g(\left[W,Y\right],W)\\
    &=&-z_2 -d_4.
\end{eqnarray*}
From this, we see that $d_3=-z_1$ and $d_4=-z_2$, which is in line with \cite{Gud-Sve-6}.

\section{Conformality}
As stated before, the foliation $\F$ is conformal if and only if Equation \eqref{confFoli} is fulfilled. We see that this is equivalent to saying that for the orthonormal basis $\{X,Y\}$ of $\H$
$$B^\H(X,X)=B^\H(Y,Y) \ \ \textrm{and} \ \ B^\H(X,Y)=0.
$$
By the use of the Koszul formula for Riemannian Lie groups \cite[Proposition 6.13]{Gud-Riemann} and Equation \eqref{ZX} to \eqref{WY} we see that
\begin{equation*}
    g(B^\H(X,X),Z)=g(X,\lb ZX)=\alpha,
\end{equation*}
\begin{equation*}
    g(B^\H(X,X),W)=g(X,\lb WX)=a,
\end{equation*}
\begin{equation*}
    g(B^\H(Y,Y),Z)=g(Y,\lb ZY)=b_2,
\end{equation*}
\begin{equation*}
    g(B^\H(Y,Y),W)=g(Y,\lb WY)=b_4,
\end{equation*}
\begin{equation*}
    g(B^\H(X,Y),Z)=\frac{1}{2}(g(\lb ZX,Y)+g(\lb ZY,X))=\frac{1}{2}(\beta+a_2) \ \ \textrm{and}
\end{equation*}
\begin{equation*}
    g(B^\H(X,Y),W)=\frac{1}{2}(g(\lb WX,Y)+g(\lb WY,X))=\frac{1}{2}(b+a_4).
\end{equation*}
It follows that
$$b_2=\alpha, \ \ b_4=a,\ \ a_2=-\beta
\ \ \textrm{and} \ \ a_4=-b.
$$
We observe that our resulting constants are the same as those in \cite{Gud-Sve-6}.

\chapter{Almost Hermitian Manifolds}\label{AlmostHermitianManifolds}
\section{Almost Complex Manifolds}
We will now define a complex manifold. Recall that a topological manifold $M$ of dimension $m$ is a paracompact Hausdorff space such that for all points $p$ in $M$, there exists a neighbourhood $U_p$ and a homeomorphism $x_p:U_p\rightarrow x(U_p)\subset \R^m$. As will be seen, the definition of a complex manifold is similarly defined:

\begin{definition}
A $2m$-dimensional manifold is called a \textit{complex manifold} if there exists an atlas
\begin{equation*}
    \A=\{(U_\alpha,z^\alpha):\alpha\in\I\},
\end{equation*}
of complex charts $$z^\alpha:U_\alpha\rightarrow z^\alpha(U_\alpha)\subset\C^m$$
such that, for each $\alpha,\beta\in\I$,
\begin{itemize}
    \item[(i)] $U_\alpha\cap U_\beta=\emptyset$ or
    \item[(ii)] the transition map
    $z^\alpha\circ (z^\beta)^{-1}:\C
    ^m\rightarrow\C^m$ is holomorphic.
\end{itemize}
We then say that $M$ is a \textit{complex manifold of dimension $m$}.
\end{definition}

Suppose that $M$ is a complex manifold of dimension $m$ and let $(U_\alpha,z^\alpha)$ be a complex chart such that $U_\alpha$ contains the point $p\in M$. Let $k$ be an integer between $1$ and $m$. The k-th coordinate of the point $p\in U_\alpha$ is then given by the map
\begin{eqnarray*}
z_k^\alpha:U_\alpha&\rightarrow&\C\\
p&\mapsto& \mathrm{proj}_k(z^\alpha(p)).
\end{eqnarray*}
Each coordinate has a decomposition $z_k^\alpha=x_k^\alpha+iy_k^\alpha$, where $x_k^\alpha$ and $y_k^\alpha$ are real. The set $$\left\{\frac{\partial}{\partial x_k^\alpha},\frac{\partial}{\partial y_k^\alpha}:1\leq k\leq m\right\}$$ then becomes a basis of the tangent space $T_pM$. Note that we now regard $M$ as a real manifold.

\begin{definition}\label{NaturalAlmostComplex}
Let $M$ be a complex $m$-dimensional manifold and let $(z_1,\dots,z_m)$ be local coordinates around a point $p\in M$, each having the real decomposition $z_k=x_k+iy_k$ for $1\leq k\leq m$. A \textit{complex structure} $J:TM\rightarrow TM$ is an endomorphism of the tangent bundle $TM$ of $M$ which in local coordinates is given by
\begin{equation}\label{NaturalAlmCom}
J_p\left(\frac{\partial}{\partial x_k^\alpha}\right)=\frac{\partial}{\partial y_k^\alpha}, \ \
J_p\left(\frac{\partial}{\partial y_k^\alpha}\right)=-\frac{\partial}{\partial x_k^\alpha}.
\end{equation}
\end{definition}

The following example shows that complex structures can be seen as generalizations of multiplication by $i$ in the complex plane.

\begin{example}
The space $\C$ of complex numbers is clearly a complex manifold. Any $z\in\C$ has a decomposition $z=x+iy$, where $x$ and $y$ are real numbers. At any point in $\C$, the tangent space is isomorphic to $\C$ itself. Thus, $e_1=1$ and $e_2=i$ become basis vectors for the tangent space at any point in $\C$. Let $J:TM\rightarrow TM$ be an endomorphism given by $J_w(z)=iz$ for $w,z\in\C$. Then $J$ is a complex structure by Equation \eqref{NaturalAlmCom}, since $J_w(e_1)=e_2$ and $J_w(e_2)=-e_1$ for any $w\in\C$.
\end{example}

\begin{definition}
An \textit{almost complex structure} on a differentiable manifold $M$ of even dimension is a tensor field of type $(1,1)$, such that for each point $p\in M$, the restriction $J_p=\left.J\right|_{T_pM}$ is an endomorhism $J_p:T_pM\rightarrow T_pM$ such that $J_p^2=-\id_{T_pM}$. If there exists an almost complex structure on $M$, the manifold $M$ is called an \textit{almost complex manifold}.
\end{definition}

\begin{remark}
Note that the complex structure $J$ on a complex manifold $M$ given in Definition \ref{NaturalAlmostComplex} is an almost complex structure. $J$ is also called the \textit{standard almost complex structure}.
\end{remark}

An almost complex structure $J$ on an almost complex manifold $(M,J)$ is called \textit{integrable} if there exists an atlas of complex charts such that any complex coordinates
satisfy \eqref{NaturalAlmCom}.

\begin{definition}\label{NijTensorDef}
The \textit{Nijenhuis tensor} of an almost complex manifold $(M,J)$ is defined as
\begin{equation}
    N_J(U,V)=
    \left[U,V\right]
    +
    J\left[JU,V\right]
    +
    J\left[U,JV\right]
    -
    \left[JU,JV\right],
\end{equation}
where $U,V\in C^\infty(TM)$.
\end{definition}

The following result is called the \textit{Newlander-Nirenberg theorem} \cite{New-Nir}.

\begin{theorem}\label{New-Nir}
For an almost complex manifold $(M,J)$, the almost complex structure $J$ is integrable if and only if its associated Nijenhuis  tensor $N_J$ vanishes for all vector fields $U,V\in C^\infty(TM)$.
\end{theorem}
\begin{proof}
See \cite{New-Nir}.
\end{proof}

\section{Almost Hermitian Manifolds}

\begin{definition}\label{AlmHerMfdDfn}
Let $(M,g,J)$ be a Riemannian manifold  equipped with an (almost) complex structure $J$. The structure $J$ is said to be \textit{compatible} with the metric $g$ if
\begin{equation*}
    g(E,F)=g(JE,JF),\qquad E,F\in C^\infty(TM).
\end{equation*}
The metric $g$ is then said to be \textit{Hermitian}. An \textit{(almost) Hermitian manifold} is a Riemannian manifold equipped with a compatible (almost) complex structure. 
\end{definition}

\begin{remark}\label{basisJ}
Suppose that $(M,g,J)$ is an (almost) Hermitian manifold of dimension $2n$ in the sense of a real manifold. Given a point $p$ in $M$, we can construct an orthonormal basis of the tangent space $T_pM$ in the following way. Let $e_1$ be an arbitrary tangent vector at $p$ of unitary length. That is, $g(e_1,e_1)=1$. Let $e_2=Je_1$. Note that $e_2$ is also of unit length since $J$ is compatible with $g$. By compatibility, $e_1$ and $e_2$ are also orthogonal since $$g(e_1,e_2)=g(e_1,Je_1)=g(Je_1,J^2e_1)=g(Je_1,-e_1)=-g(e_2,e_1),$$ 
or equivalently, $g(e_1,e_2)=0$. The tangent vectors $e_1$ and $e_2$ generate a plane $V_1$. In $V_1^\perp$, we pick a tangent vector $e_3$ of unit length and let $e_4=Je_3$. Then $\{e_1,e_2,e_3,e_4\}$ generates a $4$-dimensional space $V_2$. Let $e_5$ be a tangent vector of unit length in the space $V_2^\perp$ and so on. Eventually, we are left with the set $\{e_1,e_2,\dots,e_{2n-1},e_{2n}\}$ of orthonormal tangent vectors where $J(e_{2k-1})=e_{2k}$. Since $J^2(e_{2k-1})=-e_{2k-1}$, we also have $J(e_{2k})=-e_{2k-1}$.
\end{remark}

\begin{remark}
Notice that an (almost) complex Riemannian manifold $(M,g,J)$ always has a Hermitian metric, since we can define a Riemannian metric $h$ by
\begin{equation*}
    h(X,Y)=g(X,Y)+g(JX,JY).
\end{equation*}
It is easy to see that $J$ is compatible with the metric $h$. 
\end{remark}

\begin{definition}\label{KahlerForm}
Let $(M,g,J)$ be an (almost) Hermitian manifold. Then its \textit{K\"{a}hler form} $\omega$ is a 2-form defined by
\begin{equation*}
    \omega(E,F)=g(JE,F),\qquad E,F\in C^\infty(TM).
\end{equation*}
\end{definition}

\begin{definition}\label{AlmKahlDef}
Let $(M,g,J)$ be an (almost) Hermitian manifold with Levi-Cevita connection $\nabla$.  The (almost) complex structure $J$ is said to be (almost) \textit{K\"{a}hler} if and only if $\nabla J=0$ i.e. if $J$ is parallel. In that case, the triple $(M,g,J)$ is called an \textit{(almost) K\"{a}hler manifold}.
\end{definition}

We will now introduce a lemma (see Lemma \ref{W-B}) which shows that the definition of a Hermitian manifold also can be defined in terms of its K\"{a}hler form. Before doing that, we need to define the exterior derivative.

\begin{definition}\label{extder}
Let $M$ be a differentiable manifold, $\omega$ be a k-form on $M$ and $X_1,\dots,X_{k+1}$ be smooth vector fields. Then the \textit{exterior derivative} $d\omega$ of $\omega$ is given by
\begin{eqnarray*}
    d\omega(X_1,\dots,X_{k+1})
    &=&
    \sum_{i=1}^{k+1}(-1)^{i+1}X_i(\omega(X_1,\dots,\hat{X_i},\dots,X_{k+1}))\\
    &&+
    \sum_{i<j}(-1)^{i+j}\omega(\left[X_i,X_j\right],X_1,\dots,\hat{X_i},\dots,\hat{X_j},\dots,X_{k+1}),
\end{eqnarray*}
where $\hat{X}_i$ means that $X_i$ is missing.
\end{definition}
\begin{example}
The K\"{a}hler form $\omega$ on an (almost) Hermitian manifold is a 2-form. By letting $k=2$ in Definition \ref{extder} we get
\begin{equation}\label{ExtDerKahlForm}
\begin{aligned}
    d\omega(X_1,X_2,X_3)
    =&
    X_1(\omega(X_2,X_3))-X_2(\omega(X_1,X_3))+X_3(\omega(X_1,X_2))\\
    &-\omega(\lb{X_1}{X_2},X_3)
    +\omega(\lb{X_1}{X_3},X_2)
    -\omega(\lb{X_2}{X_3},X_1).
\end{aligned}
\end{equation}
\end{example}

The next lemma is taken from \cite[Proposition 4.16]{Ballman}.

\begin{lemma}\label{W-B}
Let $(M,g,J)$ be an (almost) Hermitian manifold. Then
\begin{equation*}
    d\omega(X,Y,Z)=
    g((\nabla_XJ)Y,Z)
    +
    g((\nabla_YJ)Z,X)
    +
    g((\nabla_ZJ)X,Y)
\end{equation*}
and
\begin{equation*}
    2\cdot g((\nabla_XJ)Y,Z)
    =
    d\omega(X,Y,Z)
    -d\omega(X,JY,JZ)
\end{equation*}
for any vector fields $X,Y,Z\in C^\infty(TM)$.
\begin{proof}
Notice that, since $J$ is compatible with $g$, $$g(JX,Y)=g(J^2X,JY)=-g(X,JY).$$ 
According to the definition of the covariant derivative $\nabla J$ of $J$, we have  $$(\nabla_XJ)Y=\nabla_X(JY)-J\nabla_XY$$ and we can expand
\begin{eqnarray*}
g((\nabla_XJ)Y,Z)&=&g(\nabla_X(JY),Z)-g(J\nabla_XY,Z)\\
&=&
g(\nabla_X(JY),Z)+g(\nabla_XY,JZ).
\end{eqnarray*}
By an application of the Koszul formula
, we get
\begin{eqnarray*}
    g(\nabla_X(JY),Z)
    &=&
    \frac{1}{2}\{X(g(JY,Z))+(JY)(g(X,Z))-Z(g(X,JY))\\
    &&+g(\lb{Z}{X},JY)+g(\lb{Z}{JY},X)
    +g(Z,\lb{X}{JY})
    \}
\end{eqnarray*}
and
\begin{eqnarray*}
    g(\nabla_XY,JZ)
    &=&
    \frac{1}{2}\{X(g(Y,JZ))+Y(g(X,JZ))-(JZ)(g(X,Y))\\
    &&+
    g(\lb{JZ}{X},Y)+g(\lb{JZ}{Y},X)+g(JZ,\lb{X}{Y})
    \}.
\end{eqnarray*}
After doing the same thing for $g((\nabla_YJ)Z,X)$ and $g((\nabla_ZJ)X,Y)$, we see that
\begin{eqnarray*}
    &&g((\nabla_XJ)Y,Z)
    +
    g((\nabla_YJ)Z,X)
    +
    g((\nabla_ZJ)X,Y)\\=&&
    X(g(JY,Z))
    -Y(g(JX,Z))
    +Z(g(JX,Y))\\
    &&
    -g(J\lb{X}{Y},Z)+g(J\lb{X}{Z},Y)-g(J\lb{Y}{Z},X)\\
    =&&
    X(\omega(Y,Z))
    -Y(\omega(X,Z))
    +Z(\omega(X,Y))\\
    &&
    -\omega(\lb{X}{Y},Z)+\omega(\lb{X}{Z},Y)-\omega(\lb{Y}{Z},X)
\end{eqnarray*}
This corresponds to the exterior derivative of $\omega$ in Equation \eqref{ExtDerKahlForm}. Thus the first statement of the lemma holds.

To prove the second part, we use the exterior derivative $d\omega$ of $\omega$ and get
\begin{eqnarray*}
    d\omega(X,Y,Z)&=&
    X(g(JY,Z))
    -Y(g(JX,Z))
    +Z(g(JX,Y))\\
    &&\quad
    -g(J\lb{X}{Y},Z)+g(J\lb{X}{Z},Y)-g(J\lb{Y}{Z},X)\\
\end{eqnarray*}
and
\begin{eqnarray*}
    -d\omega(X,JY,JZ)=&&
    -X(g(-Y,JZ))
    +JY(g(X,Z))
    -JZ(g(X,Y))\\
    &&\quad
    +g(\lb{X}{JY},Z)-g(\lb{X}{JZ},Y)+g(J\lb{JY}{JZ},X).
\end{eqnarray*}
After adding and rearranging, we get
\begin{eqnarray*}
    d\omega(X,Y,Z)-d\omega(X,JY,JZ)&=&
    -Y(g(JX,Z))+JY(g(X,Z))
    \\
    &&
    \quad+Z(g(JX,Y))-JZ(g(X,Y)) \\
    &&
    \quad+g(J\lb{JY}{JZ},X)-g(J\lb{Y}{Z},X)
    \\
    &&
    \quad+g(J\lb{X}{Z},Y)-g(\lb{X}{JZ},Y)\\
    &&
    \quad+g(\lb{X}{JY},Z)-g(J\lb{X}{Y},Z).
\end{eqnarray*}
By our earlier calculations we find that
\begin{eqnarray*}
    2(g(\nabla_X(JY),Z)+g(\nabla_XY,JZ))
    &=&
    -Y(g(JX,Z))+JY(g(X,Z))\\
    &&\quad+Z(g(JX,Y))-JZ(g(X,Y))\\
    &&\quad+g(\lb{Z}{JY},X)+g(\lb{JZ}{Y},X)\\
    &&\quad+g(J\lb{X}{Z},Y)-g(\lb{X}{JZ},Y)\\
    &&\quad+g(Z,\lb{X}{JY})-g(Z,J\lb{X}{Y}).
\end{eqnarray*}
From this we see that
\begin{equation*}
2(g(\nabla_X(JY),Z)+g(\nabla_XY,JZ))=(d\omega(X,Y,Z)-d\omega(X,JY,JZ))
\end{equation*}
and the second statement holds.
\end{proof}
\end{lemma}

As a direct consequence of Lemma \ref{W-B}, we now get the following:
\begin{proposition}[{\cite[Theorem 4.17]{Ballman}}]\label{domegaKahler}
Let $(M,g,J)$ be an (almost) Hermitian manifold and denote by $\omega$ its corresponding K\"{a}hler form. Then
$\nabla J=0$ if and only if $d\omega=0$, where $d\omega$ is the exterior derivative of $\omega$.
\end{proposition}

\chapter{The Space of Covariant Derivatives of the K\"{a}hler Form}\label{SpaceW}

We define a finite-dimensional vector space $W$ consisting of tensors with the same symmetries as those of the covariant derivative of the K\"{a}hler form of an almost Hermitian manifold. This work is based on the article \cite{Gra-Her}.

We consider a real vector space $V$ of dimension $2n$. $V$ is equipped with a positive definite inner product $g$ and an almost complex structure $J$ which is compatible with $g$, meaning $g(Jx,Jy)=g(x,y)$ for all $x,y\in V$.
Put
 \begin{equation}
    W=\{\alpha\in  V^*\otimes  V^*\otimes  V^*
    \,|\,
    \alpha(x,y,z)=-\alpha(x,z,y)=-\alpha(x,Jy,Jz), x,y,z\in V
    \},
\end{equation}
where $ V^*$ is the dual space of $V$.
We define $\Bar{\alpha}(z)\in V^*$ by
\begin{equation*}
    \Bar{\alpha}(z)=\sum_{i=1}^{2n}\alpha(e_i,e_i,z),
\end{equation*}
where $\{e_1,\dots,e_{2n}\}$ is an orthonormal basis for $ V$. We also define a scalar product $\langle\cdot,\cdot\rangle$ on $W$ by
\begin{equation}\label{Winnerproduct}
    \langle\alpha,\beta\rangle
    =
    \sum_{i,j,k=1}^{2n}\alpha(e_i,e_j,e_k)\beta(e_i,e_j,e_k),
\end{equation}
where $\alpha,\beta\in W$ and $\{e_1,\dots,e_{2n}\}$ is an orthonormal basis for $ V$. 
The space $W$ can then be split into four different subsets\footnote{Note that $W_4$ is defined differently than in \cite{Gra-Her}, due to a mistake found by Kexing Chen, author of \cite{Che}.} \cite{Gra-Her}
\begin{eqnarray*}
W_1&=&\{\alpha\in W\,|\,\alpha(x,x,z)=0\textrm{ for all }x,z\in V\},\\
W_2&=&\{\alpha\in W\,|\,\alpha(x,y,z)+\alpha(z,x,y)+\alpha(y,z,x)=0\textrm{ for all }x,y,z\in V\},\\
W_3&=&\{\alpha\in W\,|\,\alpha(x,y,z)-\alpha(Jx,Jy,z)=\Bar{\alpha}(z)=0\textrm{ for all }x,y,z\in V\},\\
W_4&=&\{\alpha\in W\,|\,\alpha(x,y,z)=
    \frac{1}{2(n-1)}(g(x,y)\Bar{\alpha}(z)-g(x,z)\Bar{\alpha}(y)\\
    &&\ \ \ \ \ \ \ \ \ \ \ \ \ \ \ \ \ \ \ \ \ \ \ \ \ \ \ \  -g(x,Jy)\Bar{\alpha}(Jz)+g(x,Jz)\Bar{\alpha}(Jy)) \textrm{ for all }x,y,z\in V\}.
\end{eqnarray*}
It is stated in \cite{Gra-Her} that $W_1$, $W_2$, $W_3$ and $W_4$ are orthogonal and that $$W=W_1\oplus W_2\oplus W_3\oplus W_4.$$
There, they use a proof involving representation theory. We will, instead, prove this by using a given orthonormal basis, see Lemma \ref{AllPlus}. Proposition \ref{Split4} in the case of $n=2$ will be of most importance for our work in the next chapters.

\begin{lemma}[\cite{Gra-Her}]\label{PlusLemma}
For the above defined subsets we have
\begin{equation*}
    W_1\oplus W_2=\{\alpha\in W\,|\,\alpha(x,y,z)+\alpha(Jx,Jy,z)=0\textrm{ for all }x,y,z\in V\}
\end{equation*}
and
\begin{equation*}
   W_3\oplus W_4 =\{\alpha\in W\,|\,\alpha(x,y,z)-\alpha(Jx,Jy,z)=0\textrm{ for all }x,y,z\in V\}.
\end{equation*}
The sets $W_1$, $W_2$, $W_3$ and $W_4$ are orthogonal.
\end{lemma}
\begin{proof}
We first show that the sets $W_1$ and $W_2$ are orthogonal. Let $\alpha_1\in W_1$ and $\alpha_2\in W_2$. Then $\alpha_1(x,y,z)=-\alpha_1(y,x,z)$ for all $x,y,z\in V$. From this, we can show that $\alpha_1(x,y,z)=\alpha_1(z,x,y)=\alpha_1(y,z,x)$. The definition of the set $W$ gives $\alpha_1(x,y,y)=\alpha_2(x,y,y)=0$ and from this we see that $\alpha_1(x,y,z)\alpha_2(x,y,z)=0$ whenever at least two of the variables are equal. We also get
\begin{eqnarray*}
    &&\alpha_1(x,y,z)\alpha_2(x,y,z)+\alpha_1(y,z,x)\alpha_2(y,z,x)+\alpha_1(z,x,y)\alpha_2(z,x,y)\\
    &=&
    \alpha_1(x,y,z)(\alpha_2(x,y,z)+\alpha_2(y,z,x)+\alpha_2(z,x,y))
    \\&=&0.
\end{eqnarray*}
Since $\alpha_1(x,y,z)\alpha_2(x,y,z)=\alpha_1(x,z,y)\alpha_2(x,z,y)$, we see that the sum of all of all the possible combinations of $\alpha_1(x,y,z)\,\alpha_2(x,y,z)$ for any distinct $x,y,z\in V$ is zero.
If $e_i\neq e_j\neq e_k$, all of its combinations show up in the sum in Equation \eqref{Winnerproduct}.
Thus $\langle\alpha_1,\alpha_2\rangle=0$. From this we see that $W_1$ and $W_2$ are orthogonal.

We next show that the sets $W_3$ and $W_4$ are orthogonal. To do this, we need to construct an orthonormal basis for $V$. By Remark \ref{basisJ} we can construct an ortonormal basis  $\{e_1,e_2,\dots,e_{2n-1},e_{2n}\}$ where $J(e_{2k-1})=e_{2k}$ and $J(e_{2k})=-e_{2k-1}$.

Let $\alpha_4\in W_4$. Then $$\alpha_4(x,y,z)=
    \frac{1}{2(n-1)}(g(x,y)\Bar{\alpha_4}(z)-g(x,z)\Bar{\alpha_4}(y)-g(x,Jy)\Bar{\alpha_4}(Jz)+g(x,Jz)\Bar{\alpha_4}(Jy))$$ for any $x,y,z\in V$. If $\alpha_3\in W$ one can show that
\begin{eqnarray*}
\langle\alpha_4,\alpha_3\rangle&=&\frac{1}{2(n-1)}\sum_{k=1}^{2n}(\Bar{\alpha_4}({e_k})\Bar{\alpha_3}(e_k)+\Bar{\alpha_4}(J{e_k})\Bar{\alpha_3}(Je_k))\\
    &&+\frac{1}{2(n-1)}\sum_{j=1}^{2n}(\Bar{\alpha_4}({e_j})\Bar{\alpha_3}(e_j)+\Bar{\alpha_4}(J{e_j})
    \Bar{\alpha_3}(Je_j))).
\end{eqnarray*}
This scalar product vanishes for any $\alpha_3\in W_3$, since $\Bar{\alpha_3}(z)=0$ for any $z\in V$ by definition. We conclude that the sets $W_3$ and $W_4$ are orthogonal.

With the above we have shown that $W_1\cap W_2=W_2\cap W_4=\{0\}$, which motivates the use of direct sums.

Let
\begin{equation*}
    U_1=\{\alpha\in W\,|\,\alpha(x,y,z)+\alpha(Jx,Jy,z)=0\textrm{ for all }x,y,z\in V\}.
\end{equation*}
We first prove that $ W_1\oplus W_2$ is a subset of $U_1$. Let $\alpha_1\in W_1$. Then $\alpha_1(x,y,z)=-\alpha_1(y,x,z)$. By using this and the definition of $W$ we get
\begin{equation*}
    \alpha_1(x,y,z)=-\alpha_1(y,x,z)=\alpha_1(y,Jx,Jy)=-\alpha_1(Jx,y,Jz)=-\alpha_1(Jx,Jy,z).
\end{equation*}

If $\alpha_2\in W_2$ then
\begin{equation*}
    \alpha_2(Jx,Jy,z)
    =
    -\alpha_2(z,Jx,Jy)-\alpha_2(Jy,z,Jx)
    =
    \alpha_2(z,x,y)-\alpha_2(Jy,Jz,x).
\end{equation*}
Thus
\begin{eqnarray*}
     \alpha_2(Jx,Jy,z)
    &=&
     \alpha_2(z,x,y)- \alpha_2(Jy,Jz,x)\\
    &=&
     \alpha_2(z,x,y)-( \alpha_2(x,y,z)- \alpha_2(Jz,Jx,y))\\
    &=&
     \alpha_2(z,x,y)- \alpha_2(x,y,z)+ \alpha_2(Jz,Jx,y)\\
    &=&
     \alpha_2(z,x,y)- \alpha_2(x,y,z)+( \alpha_2(y,z,x)- \alpha_2(Jx,Jy,z))\\
    &=&
    ( \alpha_2(z,x,y)+ \alpha_2(y,z,x)+ \alpha_2(x,y,z))\\
    &&-2 \alpha_2(x,y,z)- \alpha_2(Jx,Jy,z)\\
    &=&
    -2 \alpha_2(x,y,z)- \alpha_2(Jx,Jy,z).
\end{eqnarray*}
After rearranging, we get
\begin{equation*}
     \alpha_2(x,y,z)=- \alpha_2(Jx,Jy,z).
\end{equation*}

Thus $ W_1\oplus W_2$ is a subset of $U_1$. We now show that $U_1$ is a subset of $W_1\oplus W_2$. This is true if for each $\alpha\in U_1$ there exists an $\alpha_1\in W_1$ and $\alpha_2\in W_2$ such that $\alpha=\alpha_1+\alpha_2$. Given $\alpha\in U_1$, we can define an $\alpha_1$ by
\begin{equation*}
    \alpha_1(x,y,z)=\frac{1}{2}(\alpha(x,y,z)-\alpha(y,x,z)),
\end{equation*}
for all $x,y,z\in V$ and an $\alpha_2$ by
\begin{equation*}
    \alpha_2(x,y,z)=\frac{1}{2}(\alpha(x,y,z)-\alpha(y,z,x)),
\end{equation*}
for all $x,y,z\in V$. One can easily check that $\alpha_1\in W_1$, $\alpha_2\in W_2$ and $\alpha=\alpha_1+\alpha_2$. Thus $U_1$ is a subset of $W_1\oplus W_2$. We conclude that
\begin{equation*}
     W_1\oplus W_2=\{\alpha\in W\,|\,\alpha(x,y,z)+\alpha(Jx,Jy,z)=0\}.
\end{equation*}

Let
\begin{equation*}
   U_2 =\{\alpha\in W\,|\,\alpha(x,y,z)-\alpha(Jx,Jy,z)=0\textrm{ for all }x,y,z\in V\}.
\end{equation*}
It is easy to see that $ W_3\oplus W_4$ is a subset of $U_2$. Let $\alpha\in U_2$. Then we can define
\begin{equation*}
    \alpha_4(x,y,z)=
    \frac{1}{2(n-1)}(g(x,y)\Bar{\alpha}(z)-g(x,z)\Bar{\alpha}(y)-g(x,Jy)\Bar{\alpha}(Jz)+g(x,Jz)\Bar{\alpha}(Jy))
\end{equation*}
and
\begin{equation*}
    \alpha_3=\alpha(x,y,z)-\alpha_4(x,y,z).
\end{equation*}
Here, a calculation shows that $\Bar{\alpha}=\Bar{\alpha}_4$ and from this it follows that $\Bar{\alpha}_3=0$. Thus $\alpha_3\in W_3$, $\alpha_4\in W_4$ and $U_2$ is a subset of $W_3\oplus W_4$. We conclude that
\begin{equation*}
    W_3\oplus W_4 =\{\alpha\in W\,|\,\alpha(x,y,z)-\alpha(Jx,Jy,z)=0\textrm{ for all }x,y,z\in V\}.
\end{equation*}

Lastly, we want to show that $ W_1\oplus W_2$ and $ W_3\oplus W_4$ are orthogonal. We again let $\{e_1,e_2,\dots,e_{2n-1},e_{2n}\}$  be a set of orthonormal vectors where $J(e_{2k-1})=e_{2k}$ and $J(e_{2k})=-e_{2k-1}$. If $\alpha\in  W_1\oplus W_2$ and $\beta\in  W_3\oplus W_4$ we get $\alpha(x,y,z)\beta(x,y,z)=-\alpha(Jx,Jy,z)\beta(Jx,Jy,z)$. This and the definition of our orthonormal basis gives
\begin{eqnarray*}
    \langle\alpha,\beta\rangle
    &=&
    \sum_{i,j,k=1}^{2n}\alpha(e_i,e_j,e_k)\beta(e_i,e_j,e_k)\\
    &=&
    \sum_{k=1}^{2n}
    \sum_{i,j=1}^{n}
    (\alpha(e_{2i},e_{2j},e_k)\beta(e_{2i},e_{2j},e_k)
    +
    \alpha(e_{2i-1},e_{2j},e_k)\beta(e_{2i-1},e_{2j},e_k)\\
    &&+
    \alpha(e_{2i},e_{2j-1},e_k)\beta(e_{2i},e_{2j-1},e_k)
    +
    \alpha(e_{2i-1},e_{2j-1},e_k)\beta(e_{2i-1},e_{2j-1},e_k
    ))\\
    &=&
    \sum_{k=1}^{2n}
    \sum_{i,j=1}^{n}
    (\alpha(e_{2i},e_{2j},e_k)\beta(e_{2i},e_{2j},e_k)
    +
    \alpha(e_{2i-1},e_{2j},e_k)\beta(e_{2i-1},e_{2j},e_k)\\
    &&
    -\alpha(e_{2i-1},e_{2j},e_k)\beta(e_{2i-1},e_{2j},e_k)
    -
    \alpha(e_{2i},e_{2j},e_k)\beta(e_{2i},e_{2j},e_k))\\
    &=&0.
\end{eqnarray*}

Thus the sets $ W_1\oplus W_2$ and $ W_3\oplus W_4$ are orthogonal. It follows that all of the sets $W_1$, $W_2$, $W_3$ and $W_4$ are orthogonal.
\end{proof}

\begin{lemma}[{\cite[Theorem 2.1]{Gra-Her}}]\label{AllPlus}
The sets $W_1$, $W_2$, $W_3$ and $W_4$ are orthogonal and
$$W=W_1\oplus W_2\oplus W_3\oplus W_4.$$
\end{lemma}
\begin{proof}
Orthogonality follows from Lemma \ref{PlusLemma}. The same lemma states that
\begin{equation*}
    W_1\oplus W_2=\{\alpha\in W\,|\,\alpha(x,y,z)+\alpha(Jx,Jy,z)=0\textrm{ for all }x,y,z\in V\}.
\end{equation*}
and
\begin{equation*}
   W_3\oplus W_4 =\{\alpha\in W\,|\,\alpha(x,y,z)-\alpha(Jx,Jy,z)=0\textrm{ for all }x,y,z\in V\}.
\end{equation*}
It is clear that $(W_1\oplus W_2)\oplus (W_3\oplus W_4)$ is a subset of $W$.

Let $\alpha\in W$. We can define $\alpha_{12}$ and $\alpha_{34}$ by
\begin{equation*}
    \alpha_{12}(x,y,z)=\frac{1}{2}(\alpha(x,y,z)-\alpha(Jx,Jy,z))
\end{equation*}
and
\begin{equation*}
    \alpha_{34}(x,y,z)=\frac{1}{2}(\alpha(x,y,z)+\alpha(Jx,Jy,z)).
\end{equation*}
Then $\alpha_{12}\in W_1\oplus W_2$, $\alpha_{34}\in W_3\oplus W_4$ and $\alpha=\alpha_{12}+\alpha_{34}$. Thus $$W=(W_1\oplus W_2)\oplus (W_3\oplus W_4)=W_1\oplus W_2\oplus W_3\oplus W_4.$$
\end{proof}

\begin{lemma}[\cite{Gra-Her}]\label{W1W3ZeroLemma}
Let $n=2$. Then $W_1=W_3=\{0\}$.
\end{lemma}
\begin{proof}

If $n=2$ the vector space $V$ is of dimension $4$. We can construct an orthonormal basis $\{X,Y,Z,W\}$ where
$$JX=Y,\ JY=-X,\ JZ=W,\ JW=-Z.$$

We will now show that $W_1=\{0\}$. Notice that if $\alpha\in W_1$, then $\alpha(x,y,z)=-\alpha(y,x,z)$ for all $x,y,z\in V$. By using this and looking at the definition of $W$, it is easy to show that
$$\alpha(x,y,z)=-\alpha(x,z,y)=-\alpha(y,x,z)=\alpha(y,z,x)=-\alpha(z,y,x)=\alpha(z,x,y)$$
for all $x,y,z\in V$. If not all $x,y,z$ are distinct, it is therefore easy to see that $\alpha(x,y,z)=0$. Let $x$, $y$ and $z$ be distinct basis vectors. The definition of the set $W$ gives $\alpha(x,y,z)=-\alpha(x,Jy,Jz)$ and we must have $x=\pm Jy$ or $x=\pm Jz$. Thus $\alpha(x,y,z)=0$ for all $x,y,z\in V$ and $W_1=\{0\}$.

Let $\beta\in W_3$. We want to show that $\beta(x,y,z)=0$ for all $x,y,z\in V$. This is true if and only if it holds for all combinations of our given basis vectors.

Since $\beta(x,y,z)=\beta(Jx,Jy,z)$ and $\beta(x,y,y)=0$ for all $x,y,z\in V$ it is clear that $\beta(x,y,z)=0$ whenever $y$ and $z$ are in the same 2-dimensional plane spanned by $\{X,Y\}$ or $\{Z,W\}$. This and the condition $\Bar{\beta}(z)=0$ for all $z\in V$ gives
\begin{eqnarray*}
\Bar{\beta}(X)&=&\beta(Z,Z,X)+\beta(W,W,X)=2\beta(Z,Z,X)=2\beta(W,W,X)=0,\\
\Bar{\beta}(Y)&=&\beta(Z,Z,Y)+\beta(W,W,Y)=2\beta(Z,Z,Y)=2\beta(W,W,Y)=0,\\
\Bar{\beta}(Z)&=&\beta(X,X,Z)+\beta(Y,Y,Z)=2\beta(X,X,Z)=2\beta(Y,Y,Z)=0,\\
\Bar{\beta}(W)&=&\beta(X,X,W)+\beta(Y,Y,W)=2\beta(X,X,W)=2\beta(Y,Y,W)=0.
\end{eqnarray*}
Thus $\beta(x,x,w)=0$ whenever $x$ and $w$ are basis vectors not both in $\{X,Y\}$ or $\{Z,W\}$. We will turn back to this later.

We mentioned earlier that $\beta(x,y,z)=0$ whenever $y$ and $z$ are in the same 2-dimensional plane spanned by $\{X,Y\}$ or $\{Z,W\}$. Suppose that $y$ and $z$ are basis vectors not both in $\{X,Y\}$ or $\{Z,W\}$. If $x$ is a basis vector, then $x$ or $\pm Jx$ must be equal to $y$ or $z$. Since
\begin{equation*}
    \beta(x,y,z)=
    -\beta(Jx,z,Jy)
    =
    -\beta(x,z,y)
    =
    \beta(Jx,y,Jz)
\end{equation*}
we can rearrange so that $\beta(x,y,z)=\pm \beta(x,x,w)$, where $x$ and $w$ are not both in $\{X,Y\}$ or $\{Z,W\}$. We showed earlier that $\beta(x,x,w)=0$ for all such basis vectors. This means that $\beta(x,y,z)=0$ for all $x,y,z\in V$. Thus $W_3=\{0\}$.
\end{proof}

As a direct consequence of our previous lemmas we get the following result.
\begin{proposition}[\cite{Gra-Her}]\label{Split4}
Suppose that $n=2$. Then
\begin{equation*}
    W=W_2\oplus W_4.
\end{equation*}
The sets $W_2$ and $W_4$ are orthogonal.
\end{proposition}

\section{The Case of 4-Dimensional Lie Groups}
Let $(G,g)$ be a Lie group of dimension $2n=4$ with a left-invariant Riemannian metric $g$. We denote by $\g$ its corresponding Lie algebra. Note that $\g$ is isomorphic to the tangent space of $G$ at the unitary element and is a real vector space of dimension $4$. Let $J$ be an almost complex structure that is compatible with the Riemannian metric $g$, meaning $g(Jx,Jy)=g(x,y)$ for all $x,y\in \g$. Thus $(G,J,g)$ is an almost Hermitian manifold.

Let $\omega$ be the K\"{a}hler form of the almost Hermitian manifold $(G,J,g)$ given by $\omega(X,Y)=g(JX,Y)$ for all $X,Y\in\g$. As stated at the beginning of the chapter, the space $W$ consists of tensors with the same symmetries as those of the covariant derivative of a K\"{a}hler form. Thus $\nabla\omega\in W$ for all almost Hermitian manifolds. The class of almost Hermitian manifolds is denoted by $\mathcal{W}$.

We denote the class of almost Hermitian manifolds satisfying $\nabla{\omega}\in W_2$ by $\mathcal{AK}$ and the class satisfying $\nabla{\omega}\in W_4$ is denoted by $\mathcal{I}$. The class of almost Hermitian manifolds satisfying $\nabla\omega\in\{0\}$ is denoted by $\mathcal{K}$.

By Proposition \ref{Split4}, $\nabla\omega\in W=W_2\oplus W_4$ for all 4-dimensional almost Hermitian manifolds. We write $\mathcal{W}=\mathcal{AK}\oplus\mathcal{I}$. We see that an almost Hermitian manifold of dimension $4$ must belong to one of the four classes $\mathcal{W}$, $\mathcal{AK}$, $\mathcal{I}$ or $\mathcal{K}$.

If $\nabla{\omega}\in W_2$ then
\begin{equation*}
    \nabla_X(\omega)(Y,Z)+\nabla_Y(\omega)(Z,X)+\nabla_Z(\omega)(X,Y)=0
\end{equation*}
for all $X,Y,Z\in \g$. By definition,
\begin{eqnarray*}
\nabla_{X}(\omega )({Y},{Z})
&=&
{X}\omega ({Y},{Z})-\omega (\nabla_{X}{Y},{Z})-\omega ({Y},\nabla_{X}{Z})\\
&=&
g (\nabla_{X}{Y},J{Z})-g (J{Y},\nabla_{X}{Z}),
\end{eqnarray*}
where we have used that ${X}g(J{Y},{Z})=0$. This can then be expanded using the Koszul formula for Riemannian Lie groups, stated in for example \cite[Proposition 6.13]{Gud-Riemann}, giving
\begin{equation}\label{nablaXYZ}
  \begin{aligned}
   \nabla_{X}(\omega )({Y},{Z})
&=&
\frac{1}{2}\{g(JZ,\lb XY)+g(\lb {JZ}X,Y)+g(X,\lb {JZ}Y)\}\\
&&
 -\frac{1}{2}\{g(JY,\lb XZ)+g(\lb {JY}X,Z)+g(X,\lb {JY}Z)\}
  \end{aligned}
\end{equation}
and we get
\begin{eqnarray*}
    \nabla_X(\omega)(Y,Z)+
    \nabla_Y(\omega)(Z,X)+
    \nabla_Z(\omega)(X,Y)\\
    =
    g(\lb XY,JZ)+g(\lb YZ,JX)+g(\lb ZX,JY).
\end{eqnarray*}
We observe that this is the same as the exterior derivative $d\omega(X,Y,Z)$ given in Definition \ref{extder} for $k=2$.
Thus $d\omega=0$. From Definition \ref{AlmKahlDef} and Proposition \ref{domegaKahler}, $\mathcal{AK}$ consists of almost K\"{a}hler manifolds.

Before characterizing our set $\I$ we give the following lemma.
\begin{lemma}\label{Idontknow}
Let $N_J$ be the Nijenhuis tensor defined in Definition \ref{NijTensorDef}. For our given 4-dimensional Lie group $(G,g)$, we have that
\begin{equation}\label{LemSe51}
g(JW,N_J(V,U))+g(JV,N_J(U,W))+g(JU,N_J(V,W))=0
\end{equation}
for all $U,V,W\in\g$ if and only if $N_J=0$.
\end{lemma}
\begin{proof}
Equation \eqref{LemSe51} clearly holds if $N_J=0$.

Suppose that $\g$ has an orthonormal basis $\{X,Y,Z,W\}$ where
$$JX=Y,\ JY=-X,\ JZ=W,\ JW=-Z.$$
Then $N_{J}(X,Y)=N_{J}(Z,W)=0$ and there exist real constants $k_1$, $k_2$, $k_3$ and $k_4$ such that
\begin{equation*}
\begin{aligned}
    N_{J}(Z,X)
    &=&
    k_1X+k_2Y+k_3 Z+k_4 W=-N_{J}(W,Y)\\
     N_{J}(Z,Y)
    &=&
    k_2 X-k_1 Y+k_4Z-k_3 W=N_{J}(W,X)
\end{aligned}
\end{equation*}
This is because $N_{J}(Z,X)=J\,N_{J}(Z,Y)$ and due to the fact that the Lie brackets become linear combinations of the base vectors.

Suppose that \eqref{LemSe51} holds for any elements in $\g$. Then
\begin{equation*}
    \begin{aligned}
    0&=&
g(J{Z},N_J({X},{Y}))+g(J{X},N_J({Y},{Z}))+g(J{Y},N_J({X},{Z}))=
2\,k_1,\\
0&=&
g(J{W},N_J({X},{Y}))+g(J{X},N_J({Y},{W}))+g(J{Y},N_J({X},{W}))
=
2\,k_2,\\
0&=&
g(J{W},N_J({X},{Z}))+g(J{X},N_J({Z},{W}))+g(J{Z},N_J({X},{W}))
=
2\,k_3,\\
0&=&
g(J{W},N_J({Y},{Z}))+g(J{Y},N_J({Z},{W}))+g(J{Z},N_J({Y},{W}))
=2\,k_4.
    \end{aligned}
\end{equation*}
Thus $k_1=k_2=k_3=k_4=0$ and we get that $N_J=0$.
\end{proof}

If $\nabla{\omega}\in W_4$ then
\begin{equation*}
    \nabla_X(\omega)(Y,Z)=\nabla_{JX}(\omega)(JY,Z)
\end{equation*}
for all $X,Y,Z\in \g$. By using Equation \eqref{nablaXYZ} we get
\begin{eqnarray*}
    0&=&2\nabla_{JX}(\omega)(JY,Z)-2\nabla_X(\omega)(Y,Z)\\
    &=&
g(JZ,N_J(Y,X))+g(JY,N_J(X,Z))+g(JX,N_J(Y,Z)),
\end{eqnarray*}
where $N_J$ is the Nijenhuis tensor defined in Definition \ref{NijTensorDef}. This is true if and only if $N_J=0$ by Lemma \ref{Idontknow}. By Theorem \ref{New-Nir}, $N_J=0$ if and only if $J$ is integrable. Thus $\mathcal{I}$ consists of almost Hermitian manifolds where $J$ is integrable.

The class $\mathcal{K}$ consists of all 4-dimensional almost Hermitian manifolds with $\nabla{\omega}\in\{0\}=W_2\cap W_4$, i.e. manifolds which are both almost K\"{a}hler and whose almost complex structure $J$ is integrable. Thus $\mathcal{K}$ is the class of K\"{a}hler manifold. Notice that $\nabla\omega=0$.

In Chapter \ref{W2W4} we return to the 4-dimensional Lie group introduced in Chapter \ref{AlmostHermitianManifolds} and determine when it belongs to the class $\mathcal{K}$, $\mathcal{AK}$ or $\mathcal{I}$.

\chapter{The Four Classes of 4-Dimensional Lie Groups}\label{W2W4}
We now return to the 4-dimensional Lie group $(G,g)$ given in Chapter \ref{4dimliegr}. Recall the orthonormal basis $\{X,Y,Z,W\}$ of the Lie algebra $\g$ of $G$. We now adapt a left-invariant almost complex structure $J$ to the decomposition $TG=\V\oplus\H$ of the tangent bundle $TG$ so that $\V$ and $\H$ are closed under $J$. Recall that $\V$ is spanned by $\{Z,W\}$ and $\H$ is spanned by $\{X,Y\}$. We get
$$JX=Y,\ JY=-X,\ JZ=W,\ JW=-Z.$$
The almost complex structure $J$ is assumed to be compatible with the Riemannian metric $g$, meaning
$$g(U,V)=g(JU,JV) \ \textrm{ for all } \ U,V\in \g.$$
This makes $(G,J,g)$ an almost Hermitian manifold by Definition \ref{AlmHerMfdDfn}. In Chapter \ref{SpaceW}, we show that $J$ must belong to one of the four classes $\mathcal{W}$, $\mathcal{AK}$, $\mathcal{I}$ or $\mathcal{K}$, where $\mathcal{W}$ consists all almost Hermitian manifolds, $\mathcal{AK}$ consists of all almost K\"{a}hler manifolds, $\mathcal{I}$ consists of all almost Hermitian manifolds where $J$ is integrable and $\mathcal{K}$ consists of all K\"{a}hler manifolds. Since $(G,J,g)$ is an almost Hermitian manifold it always belongs to the class $\mathcal{W}$. We will now determine the necessary conditions for $(G,J,g)$ to belong to the class $\mathcal{AK}$, $\mathcal{I}$ or $\mathcal{K}$.

\begin{remark}
The almost complex structure $J$ is the same as the one denoted by $J_1$ in \cite{Gud-Sve-6}. There, they claim that there are exactly two invariant almost complex structures, namely $J$ and $J_2$, where $J_2$ is defined by
$$J_2\,X=Y,\ J_2\,Y=-X,\ J_2\,Z=-W,\ J_2\,W=Z.$$
We will, however, not look at the almost complex structure $J_2$ since it is simply obtained by looking at $\{-Z,-W\}$ instead of $\{Z,W\}$ for the vertical distribution.
\end{remark}

\section{Integrability}
The Lie group $(G,J,g)$ belongs to $\mathcal{I}$ if and only if $J$ is integrable. By Theorem \ref{New-Nir}, $J$ is integrable if and only if $N_J=0$, where $N_J$ is the Nijenhuis tensor given by Definition \ref{NijTensorDef}. In our case,
\begin{equation}
    N_J(U,V)=
    \left[U,V\right]
    +
    J\left[JU,V\right]
    +
    J\left[U,JV\right]
    -
    \left[JU,JV\right]
\end{equation}
for all $U,V\in\g$. $N_J=0$ if and only if $N_J(U,V)=0$ for all basis vectors $U,V$. Notice that $N_J(U,V)=0$ when $U$ and $V$ are both in $\{X,Y\}$ or $\{Z,W\}$. Since $N_J$ is also anti-symmetric, we only need to look at the cases when $U\in\{Z,W\}$ and $V\in\{X,Y\}$. A quick calculation involving the definition of $J$ shows that
$$N_J(Z,X)=J\,N_J(Z,Y)=-N_J(W,Y)=J\,N_J(W,X).$$
By using the definition of $J$ and the obtained Lie bracket relations in Equation \eqref{LieBraRela}, we get
\begin{eqnarray*}
    N_{J}(Z,X)&=&
    \left[Z,X\right]
    +
    J\left[JZ,X\right]
    +
    J\left[Z,JX\right]
    -
    \left[JZ,JX\right]\\
    &=&
    \left[Z,X\right]
    +
    J\left[W,X\right]
    +
    J\left[Z,Y\right]
    -
    \left[W,Y\right]\\
    &=&
    \cancel{\alpha X}+\cancel{\beta Y}+z_1 Z+w_1 W\\
    &&
    +
    \cancel{a Y}-\cancel{b X}+z_3 W+z_1 Z\\
    &&
    -\cancel{\beta Y}-\cancel{\alpha X}+z_2 W-w_2 Z\\
    &&
    +\cancel{b X}-\cancel{a Y}-z_4 Z+z_2 W\\
    &=&
(2z_1-z_4-w_2)Z+(2z_2+z_3+w_1)W.
\end{eqnarray*}
Thus
\begin{equation}
   2z_2 +z_3 +w_1 =2z_1 -z_4 -w_2 =0
\end{equation}
if and only if $J$ is integrable. This is the same condition as in \cite{Gud-Sve-6} for $J_1$.

\section{Almost K\"{a}hler}
Our almost Hermitian manifold $(G,J,g)$ belongs to the class $\mathcal{AK}$ if and only if it is almost K\"{a}hler.

By Definition \ref{KahlerForm}, the K\"{a}hler form $\omega$ is given by
\begin{equation}\label{AlmKahlOmega}
    \omega(X_1,X_2)=g(JX_1,X_2),
\end{equation}
where $X_1,X_2\in\g$. By Definition \ref{AlmKahlDef} and Proposition \ref{domegaKahler}, $(G,J,g)$ is almost K\"{a}hler if and only if $d\omega=0$, where $d\omega$ is the exterior derivative of $\omega$ given by Definition \ref{extder}. By Example \ref{ExtDerKahlForm} and the definition of the K\"{a}hler form we get
\begin{equation}\label{domega}
    d\omega(X_1,X_2,X_3)
    =
    -
    g(J\left[X_1,X_2\right],X_{3})
    -g(J\left[X_2,X_3\right],X_1)
    -g(J\left[X_3,X_1\right],X_2),
\end{equation}
where $X_1,X_2,X_3\in\g$ are left-invariant vector fields. Note that $X_1(g(X_2,X_3))$ vanishes for $X_1,X_2,X_3$ left-invariant vector fields. 

We have $d\omega=0$ if and only if $d\omega(X_1,X_2,X_3)=0$ for all basis vectors $X_1$, $X_2$ and $X_3$. Clearly, 
$$d\omega(X_1,X_2,X_3)=d\omega(X_2,X_3,X_1)=d\omega(X_3,X_1,X_2)$$ and $$d\omega(X_1,X_2,X_3)=-d\omega(X_2,X_1,X_3)$$
for all $X_1,X_2,X_3\in\g$.
From this, it follows that $d\omega(X_1,X_2,X_3)=0$ if at least two of the vector fields are equal. We conclude that we only need to look at cases for distinct basis vectors and disregard how they are arranged.

By using the definition for $J$, the obtained expressions of the Lie brackets in  Equation \eqref{LieBraRela} and Equation \eqref{domega}, we get
\begin{eqnarray*}
    d\omega(X,Y,Z)
    &=&
    -g(J\left[X,Y\right],Z)
    -g(J\left[Y,Z\right],X)
    -g(J\left[Z,X\right],Y)
    \\
    &=&
    g(J(rX+\theta_1 Z+\theta_2 W),Z)\\
    &&+
    g(J(-\beta X+\alpha Y+z_2 Z+w_2 W),X)\\
    &&-
    g(J(\alpha X +\beta Y+z_1 Z+w_1 W),Y)\\
    &=&
    g(rY+\theta_1 W-\theta_2 Z,Z)\\
    &&+
    g(-\beta Y-\alpha X+z_2 W-w_2 Z,X)\\
    &&-
    g(\alpha Y -\beta X+z_1 W-w_1 Z,Y)\\
    &=&
    -\theta_2
    -
    2\,\alpha.
\end{eqnarray*}
Similarly,
$$
    d\omega(X,Y,W)
    =
    \theta_1
    -2\,a,\ \
    d\omega(X,Z,W)
    =
    0 \ \ \text{and} \ \
    d\omega(Y,Z,W)
    =
    0.
$$
Thus $d\omega=0$ if and only if $\theta_1=2a$ and $\theta_2=-2\alpha$. 

\section{K\"{a}hler}
Recall from Chapter \ref{SpaceW} that $\K$ consists of K\"{a}hler manifolds and is given by the intersection of $\mathcal{I}$ and $\mathcal{AK}$. From our previous calculations we see that $(G,J,g)$ belongs to $\K$ if and only if $$\theta_1-2a=\theta_2+2\alpha=0 \ \ \textrm{and} \ \ 2z_2+z_3+w_1=2z_1-z_4-w_2=0.$$ The space $\K$ is also defined by $\nabla{\omega}=0$. We will quickly show that we can obtain the same result by computing $\nabla{\omega}=0$.

Let $X_1,X_2,X_3\in\g$. Then Equation \eqref{nablaXYZ} gives
\begin{eqnarray*}
   \nabla_{X_1}(\omega )({X_2},{X_3})
&=&
\frac{1}{2}\{g(JX_3,\lb {X_1}{X_2})+g(\lb {JX_3}{X_1},X_2)+g(X_1,\lb {JX_3}{X_2})\}\\
&&
 -\frac{1}{2}\{g(JX_2,\lb {X_1}{X_3})+g(\lb {JX_2}{X_1},X_3)+g(X_1,\lb {JX_2}{X_3})\}.
\end{eqnarray*}
One can easily see that
\begin{eqnarray*}
\nabla_{X_1}(\omega )({X_3},{X_2})
=-\nabla_{X_1}(\omega )({X_2},{X_3}).
\end{eqnarray*}
From this it follows that $\nabla_{X_1}(\omega )({X_3},{X_2})$ vanishes if $X_2$ and $X_3$ are equal.

We find that 
\begin{equation*}
\begin{aligned}
\nabla_{Y}(\omega )({X},{Z})&=-\nabla_{X}(\omega )({Y},{Z})
=
-\nabla_{X}(\omega )({X},{W})
=
-\nabla_{Y}(\omega )({Y},{W})\\
&=
\frac{1}{2}\theta_2+\alpha,
\end{aligned}
\end{equation*}
\begin{equation*}
\begin{aligned}
\nabla_{X}(\omega)({Y},{W})&=-\nabla_{Y}(\omega )({X},{W})
=
-\nabla_{X}(\omega )({X},{Z})
=
-\nabla_{Y}(\omega )({Y},{Z})\\
&=
\frac{1}{2}\theta_1-a,
\end{aligned}
\end{equation*}
\begin{equation*}
\begin{aligned}
\nabla_{Z}(\omega )({X},{W})&=\nabla_{W}(\omega )({X},{Z})
=
-\nabla_{Z}(\omega )({Z},{Y})
=
\nabla_{W}(\omega )({W},{Y})\\
&=
-\frac{1}{2}(2z_1-z_4-w_2)
\end{aligned}
\end{equation*}
and
\begin{equation*}
\begin{aligned}
\nabla_{Z}(\omega)({Y},{W})&=\nabla_{W}(\omega )({Y},{Z})
=
-\nabla_{W}(\omega )({W},{X})
=
\nabla_{Z}(\omega )({Z},{X})\\
&=
-\frac{1}{2}(2z_2+z_3+w_1),
\end{aligned}
\end{equation*}
where we again have used the definition of $J$ and the Lie bracket relations in  Equation \eqref{LieBraRela}. Thus, $\nabla\omega=0$ if and only if 
$$\theta_1-2a=\theta_2+2\alpha=0 \ \ \textrm{and} \ \ 2z_2+z_3+w_1=2z_1-z_4-w_2=0,$$
which we from before know means $(G,J,g)$ is almost K\"{a}hler and integrable.

\chapter{The 20 Families of Lie Algebras}\label{23families}
In Chapter \ref{4dimliegr}, we introduce the Riemannian manifold $(G,g)$, where $G$ is a 4-dimen-sional Lie group and $g$ is a left-invariant Riemannian metric. The Lie algebra of $G$ is denoted by $\g$. We let $K$ be a $2$-dimensional subgroup of $G$ and denote by $\k$ the Lie algebra of $K$. $\m$ is the orthogonal complement of $\k$ with respect to the metric $g$. We then introduce an orthonormal basis $\{X,Y,Z,W\}$ of $\g$ such that $\{Z,W\}$ is an orthonormal basis of $\k$ and $\{X,Y\}$ is an orthonormal basis of $\m$. $\k$ generates an involutive distribution $\V$ and $\m$ generates its orthogonal distribution $\H$. By $\F$, we denote the foliation of codimension $2$ associated to $\V$. We show in Chapter \ref{4dimliegr} that if $\F$ is minimal and conformal, $\g$ must have the Lie bracket relations
\begin{eqnarray*}
\lb WZ&=&\lambda W,\\
\lb ZX&=&\alpha X +\beta Y+z_1 Z+w_1 W,\\
\lb ZY&=&-\beta X+\alpha Y+z_2 Z+w_2 W,\\
\lb WX&=&     a X     +b Y+z_3 Z-z_1W,\\
\lb WY&=&    -b X     +a Y+z_4 Z-z_2W,\\
\lb YX&=&     r X         +\theta_1 Z+\theta_2 W
\end{eqnarray*}
with real coefficients. These relations alone do not necessarily describe a Lie algebra, since a Lie algebra must also fulfil the Jacobi identity. It has been proved in \cite{Gud-Sve-6} that $\g$ must belong to one of $20$ families $\g_1,\dots,\g_{20}$ of Lie algebras specified below.

 Recall the almost complex structure $J$ given by $$JX=Y,\ JY=-X,\ JZ=W,\ JW=-Z.$$ In Chapter \ref{W2W4} we show that $J$ is integrable ($\I$) if and only if $$2z_2+z_3+w_1=2z_1-z_4-w_2=0$$ and almost K\"{a}hler ($\mathcal{AK}$) if and only if $$\theta_1-2a=\theta_2+2\alpha=0.$$ 
The structure $J$ is K\"{a}hler ($\K$) if and only if it is both almost K\"{a}hler and integrable.

In this chapter, we go through all of the 20 families obtained in \cite{Gud-Sve-6} and determine the conditions for each family to be integrable, almost K\"{a}hler or K\"{a}hler. The results are summarized in Table \ref{tab:J1}.

\begin{table}[!h]
 \centering
 
\begin{tabular}{ |  c| C{3cm} | C{3cm}| C{3cm}| } 
  \hline
  Family   & Almost K\"{a}hler, $d\omega=0$ & Integrability, $N_{J}=0$ & K\"{a}hler  \\ 
  \hline \hline
  $\g_1$ &  $w_1=0$ & $w_1=w_2=0$ & $w_1=w_2=0$ \\ 
  \hline
  $\g_2$ & $\alpha=0$ & $w_1=w_2=0$  & $\alpha=w_1=w_2=0$ \\ 
  \hline
  $\g_3$ &  $\theta_2=-2\alpha\neq 0$ & $w_1=w_2=0$  & $\theta_2=-2\alpha\neq 0$ and $w_1=w_2=0$  \\ 
  \hline
  $\g_4$  & $2\lambda^2=z_2w_1$ & $2z_2+w_1=w_2=0$ & never true \\ 
  \hline
  $\g_5$  & $r^2=4(\alpha b-a\beta)$ & $a=\beta$ and $b=-\alpha$ & never true \\ 
  \hline
  $\g_6$  & $\theta_1=\theta_2=0$ & 
  $r={(z_1^2+z_3^2)}/{z_3}$ and $z_2={(z_1^2-z_3^2)}/{2z_3}$ & $\theta_1=\theta_2=0$, $r={(z_1^2+z_3^2)}/{z_3}$ and $z_2={(z_1^2-z_3^2)}/{2z_3}$
  \\ 
  \hline
  $\g_7$  & $\theta_1=\theta_2=0$ & $2z_2+w_1=w_2=0$ & $\theta_1=\theta_2=0$ and $2z_2+w_1=w_2=0$\\ 
  \hline
  $\g_8$  & $\theta_1=\theta_2=0$ & $2z_2=z_4+w_2=0$ & $\theta_1=\theta_2=0$ and $2z_2=z_4+w_2=0$\\ 
  \hline
  $\g_9$  & $\theta_1=\theta_2=0$ & $z_3-r=z_4=0$ & $\theta_1=\theta_2=0$ and $z_3-r=z_4=0$ \\ 
  \hline
  $\g_{10}$  & never true & always true & never true \\ 
  \hline
  $\g_{11}$  & $\theta_1=\theta_2=0$ & never true & never true\\ 
  \hline
  $\g_{12}$  & $\theta_1=\theta_2=0$ & $z_3=-w_1\neq 0$ & $\theta_1=\theta_2=0$ and $z_3=-w_1\neq 0$\\ 
  \hline
  $\g_{13}$  & $\theta_1=\theta_2=0$ & never true & never true \\ 
  \hline
  $\g_{14}$  & $\theta_1=\theta_2=0$ & $z_2=z_4+w_2=0$ & $\theta_1=\theta_2=0$ and $z_2=z_4+w_2=0$\\ 
  \hline
  $\g_{15}$  & never true & $w_1=w_2=0$ & never true \\
  \hline
  $\g_{16}$  & $\theta_1=\theta_2=0$ & $w_1=w_2=0$ & $\theta_1=\theta_2=0$ and $w_1=w_2=0$\\
  \hline
  $\g_{17}$  & never true & $w_1=w_2=0$ & never true \\ 
  \hline
  $\g_{18}$  & $\theta_1=\theta_2=0$ & $z_3=z_4=0$ & $\theta_1=\theta_2=0$ and $z_3=z_4=0$\\ 
  \hline
  $\g_{19}$  & never true & $w_1=w_2=0$ & never true \\ 
  \hline
  $\g_{20}$  & never true & $w_1=w_2=0$ & never true \\ 
  \hline
 \end{tabular} 
\caption[]{Conditions for each Lie algebra family $\g_1,\dots,\g_{20}$ of $(G,J,g)$ to be almost K\"{a}hler, integrable or K\"{a}hler.}
 \label{tab:J1}
\end{table}

\begin{remark}\label{HyperbolicRmrk}
The $2$-dimensional Lie group $\k$ with orthonormal basis $\{Z,W\}$ is flat if $\lambda=0$ and hyperbolic if $\lambda\neq0$. We see this by looking at the sectional curvature defined in for example \cite[Definition 8.14]{Gud-Riemann}. Note that $Z$ and $W$ are orthonormal basis vectors. Let $$\sigma=\mathrm{span}_\R\{Z,W\}.$$
Then the sectional curvature at the identity becomes
$$K(\sigma)=g(R(Z,W)W,Z),$$
where $R$ is the Riemann curvature operator. We get (see \cite[Theorem 8.9]{Gud-Riemann})
$$R(Z,W)W=\nabla_Z\nabla_WW-\nabla_W\nabla_ZW-\nabla_{\lb ZW}W.$$
By using the Koszul formula in \cite[Proposition 6.13]{Gud-Riemann} we see that
$$\nabla_WW=-\lambda Z, \ \nabla_ZZ=0,\ \nabla_ZW=0, \ \textrm{and}\ \nabla_WZ=\lambda W.$$
From this, the sectional curvature becomes
$$K(\sigma)=-\lambda^2.$$
Thus, the 2-dimensional Lie algebra $\k$ either corresponds to the flat Euclidean plane $\R^2$ or the hyperbolic disk $H^2_\lambda$.
\end{remark}

\begin{remark}\label{SemiDirect}
If $\theta_1=\theta_2=0$, Remark \ref{HyperbolicRmrk} applies to the $2$-dimensional space $\m$ with orthonormal basis $\{X,Y\}$. We also get that $\m$ becomes a Lie algebra. If, in addition, either of the Lie algebras $\k$ or $\m$ is ideal in $\g$ (meaning $\lb{\k}{\g}\subset\k$ or $\lb{\m}{\g}\subset\m$), we get a \textit{semidirect product} of Lie algebras and write $\g=\k\ltimes\m$. For more details, see \cite[Chapter 1]{Knapp}.
\end{remark}

\section{Case (A) - ($\lambda\neq 0$ and $(\lambda-\alpha)^2+\beta^2\neq 0$)}
\begin{example}[$\g_{1}(\lambda,r,w_1,w_2)$]\label{exam-A1}
This is a 4-dimensional family obtained by letting $r\neq 0$. By \cite{Gud-Sve-6}, this gives  $\alpha=\beta=0$ and $rw_1=\lambda\theta_2$, resulting in a family of \textit{solvable} Lie algebras given by the Lie bracket relations
\begin{eqnarray*}
\lb WZ&=&\lambda W,\\
\lb ZX&=&w_1 W,\\
\lb ZY&=&w_2 W,\\
\lb YX&=&r X+\frac{rw_1}{\lambda}W.
\end{eqnarray*}

\begin{itemize}
    \item[$\mathcal{AK}:$]
    For $J$ to be almost K\"{a}hler in this family, we must have that $\theta_1-2a=\theta_2+2\alpha=0$. Since $a=\theta_1=\alpha=0$, we only need to consider $\theta_2=0$. But $\lambda\theta_2=rw_1$, so this implies that $rw_1=0$. $r$ is chosen to be nonzero, so we conclude that the family is almost K\"{a}hler if and only if $w_1=0$. We get a 3-dimensional family $\g_1^{\mathcal{AK}}
(\lambda,r,w_2)$ given by
\begin{eqnarray*}
\lb WZ&=&\lambda W,\\
\lb ZY&=&w_2 W,\\
 \lb YX&=& r X.
\end{eqnarray*}
Here, $\m$ is a Lie algebra and $\k$ is ideal in $\g$. By Remark \ref{HyperbolicRmrk} and \ref{SemiDirect} the family corresponds to a semidirect product $H_\lambda^2\ltimes H_r^2$ of hyperbolic disks.
    \item[$\mathcal{I}:$]
    $J$ is integrable if and only if $2z_2+z_3+w_1=2z_1-z_4-w_2=0.$ Since $$z_1=z_2=z_3=z_4=0$$ here, we get the condition $w_1=w_2=0$ and obtain a 2-dimensional family $\g_1^{\mathcal{I}}
(\lambda,r)$ given by
\begin{eqnarray*}
\lb WZ&=&\lambda W,\\
\lb YX&=& r X.
\end{eqnarray*}

In view of Remark \ref{HyperbolicRmrk}, we notice that this family corresponds to a product of two hyperbolic disks $H_\lambda^2$ and $H_r^2$.

    \item[$\mathcal{K}:$] 
    The family $\g_1^{\mathcal{K}}(\lambda,r)=\g_1^{\mathcal{I}}(\lambda,r)$ is contained in $\g_1^{\mathcal{AK}}(\lambda,r,w_2)$, thus K\"{a}hler.
\end{itemize}

\end{example}

\begin{example}[$\g_{2}(\lambda,\alpha,\beta,w_1,w_2)$]\label{exam-A2}
Here, $r=\theta_1=\theta_2=0$, resulting in a 5-dimensional family of \textit{solvable} Lie algebras with the Lie bracket relations
\begin{eqnarray*}
\lb WZ&=&\lambda W,\\
\lb ZX&=&\alpha X+\beta Y+w_1 W,\\
\lb ZY&=&-\beta X+\alpha Y+w_2 W.
\end{eqnarray*}

\begin{itemize}
    \item[$\mathcal{AK}:$] 
    Since $\theta_1=\theta_2=a=0$, this family is almost K\"{a}hler if and only if $\alpha=0.$ This gives a 4-dimensional family $\g_{2}^{\mathcal{AK}}(\lambda,\beta,w_1,w_2)$ with the Lie bracket relations
\begin{eqnarray*}
\lb WZ&=&\lambda W,\\
\lb ZX&=&\beta Y+w_1 W,\\
\lb ZY&=&-\beta X+w_2 W.
\end{eqnarray*}
    \item[$\mathcal{I}:$] 
    Since $z_1=z_2=z_3=z_4=0$, $J$ is integrable if and only if $w_1=w_2=0,$ which gives a 3-dimensional family
    $\g_{2}^{\mathcal{I}}(\lambda,\alpha,\beta)$ given by
\begin{eqnarray*}
\lb WZ&=&\lambda W,\\
\lb ZX&=&\alpha X+\beta Y,\\
\lb ZY&=&-\beta X+\alpha Y.
\end{eqnarray*}
    \item[$\mathcal{K}:$]
    The family is K\"{a}hler if and only if it is both almost K\"{a}hler and integrable i.e. $\alpha=w_1=w_2=0.$
    We get a 2-dimensional family
    $\g_{2}^{\mathcal{K}}(\lambda,\beta)$
    given by
    \begin{eqnarray*}
\lb WZ&=&\lambda W,\\
\lb ZX&=&\beta Y,\\
\lb ZY&=&-\beta X.
\end{eqnarray*}
This corresponds to a semidirect product $H^2_\lambda\ltimes\R^2$.
\end{itemize}

\end{example}

\begin{example}[$\g_{3}(\alpha,\beta,w_1,w_2,\theta_2)$]\label{exam-A3}
Here $r=\theta_1=0,$ $\theta_2\neq 0$ and $\lambda=-2\alpha$, resulting in a 5-dimensional family of \textit{solvable} Lie algebras with the relations
\begin{eqnarray*}
\lb WZ&=&-2\alpha W,\\
\lb ZX&=&\alpha X+\beta  Y+w_1 W,\\
\lb ZY&=&-\beta  X+\alpha Y+w_2 W,\\
\lb YX&=&\theta_2W.
\end{eqnarray*}
\begin{itemize}
    \item[$\mathcal{AK}:$]
    Here $a=\theta_1=0,$ $\lambda=-2\alpha\neq0$ and $\theta_2\neq0.$ Thus, the family is almost K\"{a}hler if and only if $\theta_2=-2\alpha=\lambda\neq 0,$ giving a family 
    $\g_{3}^{\mathcal{AK}}(\alpha,\beta,w_1,w_2)$ with
\begin{eqnarray*}
\lb WZ&=&-2\alpha W,\\
\lb ZX&=&\alpha X+\beta  Y+w_1 W,\\
\lb ZY&=&-\beta  X+\alpha Y+w_2 W,\\
\lb YX&=&-2\alpha W.
\end{eqnarray*}
    \item[$\mathcal{I}:$] Since $z_1=z_2=z_3=z_4=0,$ we see that $J$ is integrable if and only if $w_1=w_2=0,$ giving a family
    $\g_{3}^{\mathcal{I}}(\alpha,\beta,\theta_2)$ with
\begin{eqnarray*}
\lb WZ&=&-2\alpha W,\\
\lb ZX&=&\alpha X+\beta  Y,\\
\lb ZY&=&-\beta  X+\alpha Y,\\
\lb YX&=&\theta_2W.
\end{eqnarray*}
    \item[$\mathcal{K}:$] The family is K\"{a}hler if and only if
$\theta_2=-2\alpha\neq 0$ and $ w_1=w_2=0.$
We get a family $\g_{3}^\mathcal{K}(\alpha,\beta)$ with
    \begin{eqnarray*}
\lb WZ&=&-2\alpha W,\\
\lb ZX&=&\alpha X+\beta  Y,\\
\lb ZY&=&-\beta  X+\alpha Y,\\
\lb YX&=&-2\alpha W.
\end{eqnarray*}
\end{itemize}

\end{example}

\section{Case (B) - ($\lambda\neq 0$ and $(\lambda-\alpha)^2+\beta^2=0$)}

\begin{example}[$\g_{4}(\lambda,z_2,w_1,w_2)$]\label{exam-B1}
This is a 4-dimensional family of \textit{solvable} Lie algebras given by the Lie bracket relations
\begin{eqnarray*}
\lb WZ&=&\lambda W,\\
\lb ZX&=&\lambda X+w_1 W,\\
\lb ZY&=&\lambda Y+z_2Z+w_2 W,\\
\lb WY&=&-z_2W,\\
\lb YX&=&-z_2 X-\frac{z_2w_1}{\lambda}W.
\end{eqnarray*}

\begin{itemize}
    \item[$\mathcal{AK}:$] Since $\theta_1=a=0,$ $\alpha=\lambda\neq0$ and $\lambda\theta_2=-z_2 w_1,$ the family is in the almost K\"{a}hler class if and only if $\theta_2=-2\alpha=-2\lambda.$ This implies that $2\lambda^2=w_1z_2$ and we get a 3-dimensional family
    $\g_{4}^{\mathcal{AK}}(\lambda,z_2,w_2)$
    with the following Lie bracket relations
\begin{eqnarray*}
\lb WZ&=&\lambda W,\\
\lb ZX&=& \lambda X+\frac{2\lambda^2}{z_2} W,\\
\lb ZY&=&\lambda Y+z_2Z+w_2 W,\\
\lb WY&=&-z_2W,\\
\lb YX&=&-z_2 X-2\lambda W.
\end{eqnarray*}
    \item[$\mathcal{I}:$] Since $$z_1=z_3=z_4=0,$$ $J$ is integrable in this family if and only if $$2z_2+w_1=w_2=0.$$ This gives a 2-dimensional family $\g_{4}^{\mathcal{I}}(\lambda,z_2)$
    with relations
\begin{eqnarray*}
\lb WZ&=&\lambda W,\\
\lb ZX&=&\lambda X-2z_2 W,\\
\lb ZY&=&\lambda Y+z_2Z,\\
\lb WY&=&-z_2W,\\
\lb YX&=&-z_2 X+\frac{2z_2^2}{\lambda}W.
\end{eqnarray*}
    \item[$\mathcal{K}:$] 
The family
can not be K\"{a}hler, since this requires $$\lambda^2+z_2^2=0,$$ which has no real solutions for $\lambda\neq 0$.
\end{itemize}

\end{example}

\section{Case (C) - ($\lambda=0$, $r\neq 0$ and $(a\beta-\alpha b)\neq 0$)}

\begin{example}[$\g_{5}(\alpha,a,\beta,b,r)$]\label{exam-C1}
This 5-dimensional family consists of \textit{solvable} Lie algebras given by the Lie bracket relations
\begin{eqnarray*}
\lb ZX&=&\alpha X +\beta Y
+\frac{r(\beta b-\alpha a)}{2(a\beta-\alpha b)} Z+\frac{r(\alpha^2-\beta^2)}{2(a\beta-\alpha b)} W,\\
\lb ZY&=&-\beta X+\alpha Y
+\frac{r(\alpha b+\beta a)}{2(a\beta-\alpha b)} Z-\frac{r\alpha\beta}{(a\beta-\alpha b)} W,\\
\lb WX&=&     a X     +b Y
+\frac{r(b^2-a^2)}{2(a\beta-\alpha b)} Z+\frac{r(\alpha a-\beta b)}{2(a\beta-\alpha b)}W,\\
\lb WY&=&    -b X     +a Y
+\frac{rab }{(a\beta-\alpha b)} Z-\frac{r(\alpha b+\beta a)}{2(a\beta-\alpha b)}W,\\
\lb YX&=&     r X
-\frac{ar^2}{2(a\beta-\alpha b)} Z+\frac{\alpha r^2}{2(a\beta-\alpha b)} W.
\end{eqnarray*}

\begin{itemize}
    \item[$\mathcal{AK}:$]
    Here, 
$$\theta_1=-\frac{ar^2}{2(a\beta-\alpha b)} \ \ \text{and} \ \ \theta_2=\frac{\alpha r^2}{2(a\beta-\alpha b)}.$$ Thus, the family is in the almost K\"{a}hler class if and only if
$$2a=-\frac{ar^2}{2(a\beta-\alpha b)} \ \ \text{and} \ \ 2\alpha=-\frac{\alpha r^2}{2(a\beta-\alpha b)}.$$
The condition $a\beta-\alpha b\neq0$ shows that not both $a$ and $\alpha$ can be $0$. If at least one is nonzero, the above gives $r^2=4{(\alpha b-a\beta)}>0$. From this we get two solutions of the form
\begin{eqnarray*}
\lb ZX&=&\alpha X +\beta Y
\mp\frac{(\beta b-\alpha a)}{\sqrt{\alpha b-a\beta}} Z
\mp\frac{(\alpha^2-\beta^2)}{\sqrt{\alpha b-a\beta}} W,\\
\lb ZY&=&-\beta X+\alpha Y
\mp\frac{(\alpha b+\beta a)}{\sqrt{\alpha b-a\beta}} Z
\pm\frac{2\alpha\beta}{\sqrt{\alpha b-a\beta}} W,\\
\lb WX&=&     a X     +b Y
\mp\frac{(b^2-a^2)}{\sqrt{\alpha b-a\beta}} Z
\mp\frac{(\alpha a-\beta b)}{\sqrt{\alpha b-a\beta}}W,\\
\lb WY&=&    -b X     +a Y
\mp\frac{2ab }{\sqrt{\alpha b-a\beta}} Z
\pm\frac{(\alpha b+\beta a)}{\sqrt{\alpha b-a\beta}}W,\\
\lb YX&=& \pm 2\sqrt{\alpha b-a\beta}\, X
+2a\, Z
-2\alpha\, W.
\end{eqnarray*}	
The solutions yield two $4$-dimensional families $\g_{5}^{\mathcal{AK},+}(\alpha,a,\beta,b)$ and $\g_{5}^{\mathcal{AK},-}(\alpha,a,\beta,b)$ for $r$ positive or negative, respectively.

    \item[$\mathcal{I}:$] Since 
$$z_1=\frac{r(\beta b-\alpha a)}{2(a\beta-\alpha b)}, \ \
z_2=\frac{r(\alpha b+\beta a)}{2(a\beta-\alpha b)}, \ \ z_3=\frac{r(b^2-a^2)}{2(a\beta-\alpha b)},$$ $$z_4=\frac{rab }{(a\beta-\alpha b)}, \ \ w_1=\frac{r(\alpha^2-\beta^2)}{2(a\beta-\alpha b)} \ \ \text{and} \ \ 
w_2=-\frac{r\alpha\beta}{(a\beta-\alpha b)},$$
we get that $J$ is integrable if and only if 
$$
      2\frac{r(\alpha b+\beta a)}{2(a\beta-\alpha b)}
+\frac{r(b^2-a^2)}{2(a\beta-\alpha b)}
+\frac{r(\alpha^2-\beta^2)}{2(a\beta-\alpha b)}=0$$
and
$$  2\frac{r(\beta b-\alpha a)}{2(a\beta-\alpha b)}-\frac{rab }{(a\beta-\alpha b)}+\frac{r\alpha\beta}{(a\beta-\alpha b)}
=0.$$
Since $r\neq 0$ and $a\beta-\alpha b\neq0$, we can cancel these out. By doing this and rearranging, we get the conditions
$$(a - \beta) (b + \alpha) = 0  \ \ \text{and} \ \ (a-\beta)^2=(b+\alpha)^2.$$
By the first condition, we must have $a=\beta$ or $b=-\alpha$. If we let $a=\beta$ in the second equation we see that we must have $b=-\alpha$ and vice versa. Thus, $J$ is integrable if and only if $a=\beta$ and $b=-\alpha$. This gives a family $\g_{5}^\mathcal{I}(\alpha,\beta,r)$ with
\begin{eqnarray*}
\lb ZX&=&\alpha X +\beta Y
-\frac{r\alpha\beta}{(\beta^2+\alpha^2)} Z+\frac{r(\alpha^2-\beta^2)}{2(\beta^2+\alpha^2)} W,\\
\lb ZY&=&-\beta X+\alpha Y
-\frac{r(\alpha^2-\beta^2)}{2(\beta^2+\alpha^2)} Z-\frac{r\alpha\beta}{(\beta^2+\alpha^2)} W,\\
\lb WX&=&     \beta X     -\alpha Y
+\frac{r(\alpha^2-\beta^2)}{2(\beta^2+\alpha^2)} Z+\frac{r\alpha \beta}{(\beta^2+\alpha^2)}W,\\
\lb WY&=&    \alpha X     +\beta Y
-\frac{r\alpha\beta }{(\beta^2+\alpha^2)} Z+\frac{r(\alpha^2-\beta^2)}{2(\beta^2+\alpha^2)}W,\\
\lb YX&=&     r X
-\frac{\beta r^2}{2(\beta^2+\alpha^2)} Z+\frac{\alpha r^2}{2(\beta^2+\alpha^2)} W.
\end{eqnarray*}
    \item[$\mathcal{K}:$]
    We showed earlier that if the family is almost K\"{a}hler, then $\alpha b-a\beta>0$.
    Suppose that $J$ is integrable. We know that this is true if and only if
    $$a=\beta \ \ \textrm{and} \ \ b=-\alpha.$$
    But then
    $$\alpha b-a\beta=-(\alpha^2+\beta^2).$$
    Since this can not be positive, the family can not also be almost K\"{a}hler. We conclude that the family can not be K\"{a}hler.
\end{itemize}

\end{example}

\section{Case (D) - ($\lambda=0$, $r\neq 0$ and $(a\beta-\alpha b)=0$)}

\begin{example}[$\g_{6}(z_1,z_2,z_3,r,\theta_1,\theta_2)$]\label{exam-D1}
In this family, $z_1^2=-w_1z_3\neq 0$ and
$$z_4 =\frac{z_3(r+2z_2)}{2z_1},\ \ w_1 = -\frac{z_1^2}{z_3},
\ \ w_2 = \frac{z_1(r-2z_2)}{2z_3}.$$
This gives a 6-dimensional family of \textit{solvable} Lie algebras where
\begin{eqnarray*}
\lb ZX&=&z_1 Z-\frac{z_1^2}{z_3} W,\\
\lb ZY&=&z_2 Z+\frac{z_1(r-2z_2)}{2z_3} W,\\
\lb WX&=&z_3 Z-z_1W,\\
\lb WY&=&\frac{z_3(r+2z_2)}{2z_1} Z-z_2W,\\
\lb YX&=&r X+\theta_1 Z+\theta_2 W.
\end{eqnarray*}

\begin{itemize}
    \item[$\mathcal{AK}:$] Since $a=\alpha=0$ the above is almost K\"{a}hler if and only if $\theta_1=\theta_2=0,$ giving a $4$-dimensional family $\g_{6}^{\mathcal{AK}}(z_1,z_2,z_3,r)$ with
    the relations
\begin{eqnarray*}
\lb ZX&=&z_1 Z-\frac{z_1^2}{z_3} W,\\
\lb ZY&=&z_2 Z+\frac{z_1(r-2z_2)}{2z_3} W,\\
\lb WX&=&z_3 Z-z_1W,\\
\lb WY&=&\frac{z_3(r+2z_2)}{2z_1} Z-z_2W,\\
\lb YX&=&r X.
\end{eqnarray*}
We see that this corresponds to a semidirect product $H_r^2\ltimes\R^2$.

    \item[$\mathcal{I}:$] We have 
$$z_1\neq0, \ \ z_4=\frac{z_3(r+2z_2)}{2z_1}, \ \ w_1=-\frac{z_1^2}{z_3} \ \ \text{and} \ \ w_2=\frac{z_1(r-2z_2)}{2z_3}.$$ Thus, $J$ is integrable if and only if
$$2z_2+z_3-\frac{z_1^2}{z_3}=0\ \ \text{and} \ \
2z_1-\frac{z_3(r+2z_2)}{2z_1}-\frac{z_1(r-2z_2)}{2z_3}=0.$$
After multiplying the first equation by $z_3$ and the second one by $z_1z_3$ and rearranging, we get that
$$z_1^2=(2z_2+z_3)z_3 \ \ \text{and} \ \
z_1^2(4z_3-r+2z_2)=z_3^2(r+2z_2).$$
By multiplying the first equation by $4z_3-r+2z_2$, we see that these two conditions correspond to
\begin{equation}
      (2z_2+z_3)(4z_3-r+2z_2)=z_3(r+2z_2),
\end{equation}
which can be rewritten as
\begin{equation}
    (z_2+z_3)(2(z_2+z_3)-r)=0.
\end{equation}
If $z_2+z_3=0$, our earlier conditions give $z_1^2+z_3^2=0$, which is not possible for nonzero real constants. Thus, we must have $z_2+z_3\neq0$. This leaves us with $r=2(z_2+z_3)$. We conclude that $J$ is integrable if and only if 
$$r=\frac{z_1^2+z_3^2}{z_3} \ \ \text{and} \ \ z_2=\frac{z_1^2-z_3^2}{2z_3},$$ giving a 4-dimensional family
$\g_{6}^\mathcal{I}(z_1,z_3,\theta_1,\theta_2)$
with relations
\begin{eqnarray*}
\lb ZX&=&z_1 Z-\frac{z_1^2}{z_3} W,\\
\lb ZY&=&\frac{z_1^2-z_3^2}{2z_3} Z+z_1 W,\\
\lb WX&=&z_3 Z-z_1W,\\
\lb WY&=&z_1 Z-\frac{z_1^2-z_3^2}{2z_3}W,\\
\lb YX&=&\frac{z_1^2+z_3^2}{z_3} X+\theta_1 Z+\theta_2 W.
\end{eqnarray*}
    \item[$\mathcal{K}:$] 
    The family is K\"{a}hler if and only if
    $$\theta_1=\theta_2=0, \ \ r=\frac{z_1^2+z_3^2}{z_3} \ \ \text{and} \ \ z_2=\frac{z_1^2-z_3^2}{2z_3},$$
    giving a 2-dimensional family
$\g_{6}^\mathcal{K}(z_1,z_3)$
with relations
\begin{eqnarray*}
\lb ZX&=&z_1 Z-\frac{z_1^2}{z_3} W,\\
\lb ZY&=&\frac{z_1^2-z_3^2}{2z_3} Z+z_1 W,\\
\lb WX&=&z_3 Z-z_1W,\\
\lb WY&=&z_1 Z-\frac{z_1^2-z_3^2}{2z_3}W,\\
\lb YX&=&\frac{z_1^2+z_3^2}{z_3} X.
\end{eqnarray*}
This corresponds to a semidirect product $H_r^2\ltimes\R^2$.
\end{itemize}

\end{example}

\begin{example}[$\g_{7}(z_2,w_1,w_2,\theta_1,\theta_2)$]\label{exam-D2}
In this family, $$z_1=z_3=z_4=0, \ \ r=2z_2 \ \ \textrm{and} \ \ w_1\neq 0,$$ resulting in a 5-dimensional family of \textit{solvable} Lie algebras of the form
\begin{eqnarray*}
\lb ZX&=&w_1 W,\\
\lb ZY&=&z_2 Z+w_2 W,\\
\lb WY&=&-z_2 W,\\
\lb YX&=&2z_2 X+\theta_1 Z+\theta_2 W.
\end{eqnarray*}

\begin{itemize}
    \item[$\mathcal{AK}:$] Since $\alpha=a=0,$ one can easily see that the family is in the almost K\"{a}hler class if and only if $\theta_1=\theta_2=0,$ which gives a 3-dimensional family
    $\g_{7}^{\mathcal{AK}}(z_2,w_1,w_2)$
    with Lie bracket relations
\begin{eqnarray*}
\lb ZX&=&w_1 W,\\
\lb ZY&=&z_2 Z+w_2 W,\\
\lb WY&=&-z_2 W,\\
\lb YX&=&2z_2 X,
\end{eqnarray*}
which correspond to a semidirect product $H_{2z_2}^2\ltimes\R^2$.
    \item[$\mathcal{I}:$] We have $z_1=z_3=z_4=0.$ Thus, $J$ is integrable if and only if $2z_2+w_1=w_2=0$ and we get a 3-dimensional family
 $\g_{7}^{\mathcal{I}}(z_2,\theta_1,\theta_2)$
 given by
\begin{eqnarray*}
\lb ZX&=&-2z_2 W,\\
\lb ZY&=&z_2 Z,\\
\lb WY&=&-z_2 W,\\
\lb YX&=&2z_2 X+\theta_1 Z+\theta_2 W.
\end{eqnarray*}
Note that $z_2\neq0$.
    \item[$\mathcal{K}:$]
    The family is K\"{a}hler if and only if
    $\theta_1=\theta_2=0,$ $w_1=-2z_2\neq0$ and $w_2=0,$
    giving a $1$-dimensional family
 $\g_{7}^{\mathcal{K}}(z_2)$ with
\begin{eqnarray*}
\lb ZX&=&-2z_2 W,\\
\lb ZY&=&z_2 Z,\\
\lb WY&=&-z_2 W,\\
\lb YX&=&2z_2 X.
\end{eqnarray*}
This corresponds to a semidirect product $H_{2z_2}^2\ltimes\R^2$.
\end{itemize}
\end{example}
\begin{example}[$\g_{8}(z_2,z_4,w_2,r,\theta_1,\theta_2)$]\label{exam-D3}
Here we have $z_1=z_3=w_1=0,$ providing a 6-dimensional family of \textit{solvable} Lie algebras with
\begin{eqnarray*}
\lb ZY&=&z_2 Z+w_2 W,\\
\lb WY&=&z_4 Z-z_2 W,\\
\lb YX&=&r X+\theta_1 Z+\theta_2 W.
\end{eqnarray*}

\begin{itemize}
    \item[$\mathcal{AK}:$]Since $\alpha=a=0,$ the family is almost K\"{a}hler if and only if $\theta_1=\theta_2=0,$ giving a $4$-dimensional family
    $\g_{8}^{\mathcal{AK}}(z_2,z_4,w_2,r)$ with
\begin{eqnarray*}
\lb ZY&=&z_2 Z+w_2 W,\\
\lb WY&=&z_4 Z-z_2 W,\\
\lb YX&=&r X.
\end{eqnarray*}
This corresponds to a semidirect product $H_{r}^2\ltimes\R^2$.    
\item[$\mathcal{I}:$]
    We have
    $z_1=z_3=w_1=0,$ so $J$ is integrable if and only if $z_2=z_4+w_2=0,$ giving a family
    $\g_{8}^{\mathcal{I}}(w_2,r,\theta_1,\theta_2)$ with
\begin{eqnarray*}
\lb ZY&=&w_2 W,\\
\lb WY&=&-w_2 Z,\\
\lb YX&=&r X+\theta_1 Z+\theta_2 W.
\end{eqnarray*}
    \item[$\mathcal{K}:$] 
    The family is K\"{a}hler if and only if
    $\theta_1=\theta_2=z_2=z_4+w_2=0.$ 
    We get a $2$-dimensional family
    $\g_{8}^{\mathcal{K}}(w_2,r)$ with
\begin{eqnarray*}
\lb ZY&=&w_2 W,\\
\lb WY&=&-w_2 Z,\\
\lb YX&=&r X.
\end{eqnarray*}
Like $\g_8^{\mathcal{AK}}$, this corresponds to a semidirect product $H_{r}^2\ltimes\R^2$.
\end{itemize}
\end{example}

\begin{example}[$\g_{9}(z_2,z_3,z_4,\theta_1,\theta_2)$]\label{exam-D4}
Here, $$z_1=w_1=w_2=0, \ \ z_3\neq0 \ \ \textrm{and} \ \ r=-2z_2,$$ which gives a 5-dimensional family of \textit{solvable} Lie algebras of the form
\begin{eqnarray*}
\lb ZY&=&z_2 Z,\\
\lb WX&=&z_3 Z,\\
\lb WY&=&z_4 Z-z_2 W,\\
\lb YX&=&-2z_2 X+\theta_1 Z+\theta_2 W.
\end{eqnarray*}

\begin{itemize}
    \item[$\mathcal{AK}:$] Since $\alpha=a=0,$ we see that the family is almost K\"{a}hler if and only if $\theta_1=\theta_2=0,$ giving
    $\g_{9}^{\mathcal{AK}}(z_2,z_3,z_4)$ with
\begin{eqnarray*}
\lb ZY&=&z_2 Z,\\
\lb WX&=&z_3 Z,\\
\lb WY&=&z_4 Z-z_2 W,\\
\lb YX&=&-2z_2 X.
\end{eqnarray*}
We see that this corresponds to a semidirect product $H_{2z_2}^2\ltimes\R^2$.
    \item[$\mathcal{I}:$] From $z_1=w_1=w_2=0,$ 
we find that $J$ is integrable if and only if $2z_2+z_3=z_4=0,$ from which we get $\g_{9}^{\mathcal{I}}(z_2,\theta_1,\theta_2)$ with the relations
\begin{eqnarray*}
\lb ZY&=&z_2 Z,\\
\lb WX&=&-2z_2 Z,\\
\lb WY&=&-z_2 W,\\
\lb YX&=&-2z_2 X+\theta_1 Z+\theta_2 W.
\end{eqnarray*}
    \item[$\mathcal{K}:$] 
    The family is K\"{a}hler if and only if
    $\theta_1=\theta_2=0,$ $z_3=-2z_2$ and $z_4=0.$
    We get a $1$-dimensional family $\g_{9}^{\mathcal{K}}(z_2)$ with
\begin{eqnarray*}
\lb ZY&=&z_2 Z,\\
\lb WX&=&-2z_2 Z,\\
\lb WY&=&-z_2 W,\\
\lb YX&=&-2z_2 X.
\end{eqnarray*}
This corresponds to a semidirect product $H_{2z_2}^2\ltimes\R^2$.
\end{itemize}
\end{example}

\section{Case (E) - ($\lambda=0$, $r=0$ and $\alpha b-a\beta\neq 0$)}

\begin{example}[$\g_{10}(\alpha,a,\beta,b)$]\label{exam-E1}
In this family,
$$z_1=z_2=z_3=z_4=w_1=w_2=\theta_1=\theta_2=0,$$
which gives a 4-dimensional family of \textit{solvable} Lie algebras with the relations
\begin{eqnarray*}
\lb ZX&=&\alpha X+ \beta Y,\\
\lb ZY&=&-\beta X+\alpha Y,\\
\lb WX&=&     a X     +b Y,\\
\lb WY&=&    -b X     +a Y.
\end{eqnarray*}
We see that this family corresponds to a semidirect product $\R^2\ltimes\R^2$.

\begin{itemize}
    \item[$\mathcal{AK}:$]Since $\theta_1=\theta_2=0,$ we see that the family is almost K\"{a}hler if and only if $a=\alpha=0.$ But then the requirement $\alpha b-a\beta\neq0$ is not true, so this family can not be almost K\"{a}hler.
    \item[$\mathcal{I}:$] Since $$z_1=z_2=z_3=z_4=w_1=w_2=0,$$
    we can see that $J$ is always integrable in this family.
    \item[$\mathcal{K}:$]
    Since this family can not be almost K\"{a}hler it can not be in the K\"{a}hler class.
\end{itemize}
\end{example}

\section{Case (F) - ($\lambda=0$, $r=0$ and $\alpha b-a\beta= 0$)}
Here, the classes are divided into disjoint cases parametrized by $\Lambda=(\alpha,a,\beta,b)$, where the variables are zero if and only if they are marked by zero. For instance, if $\Lambda=(\alpha,0,0,0)$, the variable $\alpha$ is non-zero while the variables $a$, $\beta$ and $b$ are zero.
\begin{example}[$\g_{11}(z_1,z_2,z_3,w_1,\theta_1,\theta_2)$]\label{exam-F1}
Here, we let $\Lambda=(0,0,0,0)$ and $z_1\neq 0$, giving a 6-dimensional family of \textit{solvable} Lie algebras with the Lie bracket relations
\begin{eqnarray*}
\lb ZX&=&z_1 Z+w_1 W,\\
\lb ZY&=&z_2 Z+\frac{z_2w_1}{z_1} W,\\
\lb WX&=&z_3 Z-z_1W,\\
\lb WY&=&\frac{z_2z_3}{z_1} Z-z_2 W,\\
\lb YX&=&\theta_1 Z+\theta_2 W.
\end{eqnarray*}

\begin{itemize}
    \item[$\mathcal{AK}:$] 
Since $\alpha=a=0,$ we see that this is in the almost K\"{a}hler class if and only if $\theta_1=\theta_2=0.$ This gives a 4-dimensional family
$\g_{11}^{\mathcal{AK}}(z_1,z_2,z_3,w_1)$
with
\begin{eqnarray*}
\lb ZX&=&z_1 Z+w_1 W,\\
\lb ZY&=&z_2 Z+\frac{z_2w_1}{z_1} W,\\
\lb WX&=&z_3 Z-z_1W,\\
\lb WY&=&\frac{z_2z_3}{z_1} Z-z_2 W.
\end{eqnarray*}
This corresponds to a semidirect product $\R^2\ltimes\R^2$.
    \item[$\mathcal{I}:$] From $$z_1\neq0, \ \ z_4=\frac{z_2 z_3}{z_1} \ \ \text{and} \ \ w_2=\frac{z_2 w_1}{z_1},$$ we see that $J$ is integrable if and only if
$$2z_2+z_3+w_1=0 \ \ \text{and} \ \
2z_1-\frac{z_2z_3}{z_1}-\frac{z_2w_1}{z_1}=0.$$
After multiplying the second equation by $z_1$ and rearranging, we get
$$
z_3+w_1=-2z_2 \ \ \text{and} \ \
2z_1^2-z_2(z_3+w_1)=0.
$$
Now, we plug in $z_3+w_1=-2z_2$ into the second equation and get
$z_1^2+z_2^2=0.$
This is, however, not possible since $z_1$ and $z_2$ are real constants and $z_1\neq0$. Thus, it is not possible for $J$ to be integrable in this family.
    \item[$\mathcal{K}:$] This family can not be K\"{a}hler since it can not be integrable.
\end{itemize}

\end{example}

\begin{example}[$\g_{12}(z_3,w_1,w_2,\theta_1,\theta_2)$]\label{exam-F2}
This family is obtained by letting $\Lambda=(0,0,0,0)$, $z_1=0$ and $w_1\neq 0$, which gives
$$z_2=0\ \ \text{and}\ \ z_4=\frac{z_3w_2}{w_1}$$
and we get a 5-dimensional family of \textit{solvable} Lie algebras with the Lie bracket relations
\begin{eqnarray*}
\lb ZX&=&w_1 W,\\
\lb ZY&=&w_2 W,\\
\lb WX&=&z_3 Z,\\
\lb WY&=&\frac{z_3w_2}{w_1} Z,\\
\lb YX&=&\theta_1 Z+\theta_2 W.
\end{eqnarray*}

\begin{itemize}
    \item[$\mathcal{AK}:$] Since $\alpha=a=0,$ this family is almost K\"{a}hler if and only if $\theta_1=\theta_2=0.$
This is a 3-dimensional family $\g^{\mathcal{AK}}_{12}(z_3,w_1,w_2)$
with the relations
\begin{eqnarray*}
\lb ZX&=&w_1 W,\\
\lb ZY&=&w_2 W,\\
\lb WX&=&z_3 Z,\\
\lb WY&=&\frac{z_3w_2}{w_1} Z.
\end{eqnarray*}
We see that this corresponds to a semidirect product $\R^2\ltimes\R^2$.
    \item[$\mathcal{I}:$] We have
$$z_1=z_2=0, \ \ z_4=\frac{z_3w_2}{w_1} \ \ \text{and} \ \ w_1\neq0.$$ We see that $J$ is integrable if and only if $$z_3+w_1=\frac{z_3w_2}{w_1}+w_2=0,$$
which is true for $z_3=-w_1\neq 0$, giving a 4-dimensional family $\g_{12}^{\mathcal{I}}(w_1,w_2,\theta_1,\theta_2)$
which has the Lie bracket relations
\begin{eqnarray*}
\lb ZX&=&w_1 W,\\
\lb ZY&=&w_2 W,\\
\lb WX&=&-w_1 Z,\\
\lb WY&=&-w_2 Z,\\
\lb YX&=&\theta_1 Z+\theta_2 W.
\end{eqnarray*}
    \item[$\mathcal{K}:$]
    The family is K\"{a}hler if and only if
    $\theta_1=\theta_2=0$ and $z_3=-w_1\neq 0,$
    which gives a $2$-dimensional family
    $\g_{12}^{\mathcal{K}}(w_1,w_2)$ with relations
    \begin{eqnarray*}
\lb ZX&=&w_1 W,\\
\lb ZY&=&w_2 W,\\
\lb WX&=&-w_1 Z,\\
\lb WY&=&-w_2 Z,
\end{eqnarray*}
corresponding to a semidirect product $\R^2\ltimes\R^2$.
\end{itemize}
\end{example}

\begin{example}[$\g_{13}(z_3,z_4,\theta_1,\theta_2)$]\label{exam-F3}
This family is obtained by letting $\Lambda=(0,0,0,0)$,  $z_1=w_1=0$ and $z_3\neq 0$, which gives $z_2=w_2=0$
and we are provided with a 4-dimensional family of \textit{nilpotent} Lie algebras with
\begin{eqnarray*}
\lb WX&=&z_3 Z,\\
\lb WY&=&z_4 Z,\\
\lb YX&=&\theta_1 Z+\theta_2 W.
\end{eqnarray*}

\begin{itemize}
    \item[$\mathcal{AK}:$]We see that, since $\alpha=a=0,$ this family is almost K\"{a}hler if and only if $\theta_1=\theta_2=0$. We get the $2$-dimensional family $\g_{13}^{\mathcal{AK}}(z_3,z_4)$ with
\begin{eqnarray*}
\lb WX&=&z_3 Z,\\
\lb WY&=&z_4 Z.
\end{eqnarray*}
This corresponds to a semidirect product $\R^2\ltimes\R^2$.
    \item[$\mathcal{I}:$] The condition
    $$z_1=z_2=w_1=w_2=0$$ gives that $J$ is integrable if and only if $z_3=z_4=0$. This contradicts $z_3\neq0$. Hence $\g_{13}$ can not be almost K\"{a}hler.
    \item[$\mathcal{K}:$] This family can not be K\"{a}hler.
\end{itemize}
\end{example}

\begin{example}[$\g_{14}(z_2,z_4,w_2,\theta_1,\theta_2)$]\label{exam-F4}
Here, they let $\Lambda=(0,0,0,0)$ and $z_1=z_3=w_1=0$,
providing a 5-dimensional family of \textit{solvable} Lie algebras with the relations
\begin{eqnarray*}
\lb ZY&=&z_2 Z+w_2 W,\\
\lb WY&=&z_4 Z-z_2 W,\\
\lb YX&=&\theta_1 Z+\theta_2 W.
\end{eqnarray*}

\begin{itemize}
    \item[$\mathcal{AK}:$] The condition
    $\alpha=a=0$ implies that this family is almost K\"{a}hler if and only if $\theta_1=\theta_2=0.$
    We get a $3$-dimensional family $\g_{14}^{\mathcal{AK}}(z_2,z_4,w_2)$ given by
\begin{eqnarray*}
\lb ZY&=&z_2 Z+w_2 W,\\
\lb WY&=&z_4 Z-z_2 W.
\end{eqnarray*}
We see that this corresponds to a semidirect product $\R^2\ltimes\R^2$.
    \item[$\mathcal{I}:$] From
    $z_1=z_3=w_1=0$ we see that $J$ is integrable if and only if $z_2=0$ and $z_4+w_2=0.$
    This gives a $3$-dimensional family
    $\g_{14}(w_2,\theta_1,\theta_2)$ with Lie bracket relations
\begin{eqnarray*}
\lb ZY&=&w_2 W,\\
\lb WY&=&-w_2 Z,\\
\lb YX&=&\theta_1 Z+\theta_2 W.
\end{eqnarray*}
    \item[$\mathcal{K}:$] 
    The family is K\"{a}hler if and only if
    $\theta_1=\theta_2=0,$ $z_2=0$ and $z_4+w_2=0.$
    We get a $1$-dimensional family $\g_{14}^{\mathcal{K}}(w_2)$ with
    \begin{eqnarray*}
\lb ZY&=&w_2 W,\\
\lb WY&=&-w_2 Z.
\end{eqnarray*}
This corresponds to a semidirect product $\R^2\ltimes\R^2$.
\end{itemize}
\end{example}

\begin{example}[$\g_{15}(\alpha,w_1,w_2)$]\label{exam-F5}
This is obtained by letting $\Lambda=(\alpha,0,0,0)$, which gives
\begin{equation*}
z_1=z_2=z_3=z_4=\theta_1=\theta_2=0.
\end{equation*}
and we get a 3-dimensional family of \textit{solvable} Lie algebras where
\begin{eqnarray*}
\lb ZX&=&\alpha X+w_1 W,\\
\lb ZY&=&\alpha Y+w_2 W.
\end{eqnarray*}

\begin{itemize}
    \item[$\mathcal{AK}:$]
    Here, $a=\theta_1=\theta_2=0$ and $\alpha\neq0,$ so we see that this family can not be almost K\"{a}hler.
    \item[$\mathcal{I}:$] The condition $z_1=z_2=z_3=z_4=0$ implies that $J$ is integrable if and only if
$w_1=w_2=0.$ We get a 1-dimensional family
$\g_{15}^{\mathcal{I}}(\alpha)$ with
\begin{eqnarray*}
\lb ZX&=&\alpha X,\\
\lb ZY&=&\alpha Y,
\end{eqnarray*}
corresponding to a semidirect product $\R^2\ltimes\R^2$.
    \item[$\mathcal{K}:$] This family can not be K\"{a}hler since it can not be almost K\"{a}hler.
\end{itemize}
\end{example}

\begin{example}[$\g_{16}(\beta,w_1,w_2,\theta_1,\theta_2)$]\label{exam-F6}
Here, they let $\Lambda=(0,0,\beta,0)$, which implies
\begin{equation*}
z_1=z_2=z_3=z_4=0.
\end{equation*}
From this we get a 5-dimensional family of Lie algebras which are \textit{not solvable} in general. They have the Lie bracket relations
\begin{eqnarray*}
\lb ZX&=&\beta Y+w_1 W,\\
\lb ZY&=&-\beta X+w_2 W,\\
\lb YX&=&\theta_1 Z+\theta_2 W.
\end{eqnarray*}

\begin{itemize}
    \item[$\mathcal{AK}:$] We have $\alpha=a=0,$ so this family is almost K\"{a}hler if and only if $\theta_1=\theta_2=0.$
    We get the $3$-dimensional family
    $\g_{16}^{\mathcal{AK}}(\beta,w_1,w_2)$ of \textit{solvable} Lie algebras with
\begin{eqnarray*}
\lb ZX&=&\beta Y+w_1 W,\\
\lb ZY&=&-\beta X+w_2 W.
\end{eqnarray*}
    \item[$\mathcal{I}:$] We have $z_1=z_2=z_3=z_4=0,$ so $J$ is integrable if and only if $w_1=w_2=0,$
    giving a $3$-dimensional family $\g_{16}^{\mathcal{I}}(\beta,\theta_1,\theta_2)$ of Lie algebras which are \textit{not solvable} in general and have the relations
\begin{eqnarray*}
\lb ZX&=&\beta Y,\\
\lb ZY&=&-\beta X,\\
\lb YX&=&\theta_1 Z+\theta_2 W.
\end{eqnarray*}
    \item[$\mathcal{K}:$]
    The family is K\"{a}hler if and only if
    $\theta_1=\theta_2=0$ and $w_1=w_2=0.$
    We get a 1-dimensional family
    $\g_{16}^{\mathcal{K}}(\beta)$ of \textit{solvable} Lie algebras
    with Lie bracket relations
    \begin{eqnarray*}
\lb ZX&=&\beta Y,\\
\lb ZY&=&-\beta X.
\end{eqnarray*}
We see that this corresponds to a semidirect product $\R^2\ltimes\R^2$.
\end{itemize}
\end{example}

\begin{example}[$\g_{17}(\alpha,a,w_1,w_2)$]\label{exam-F7}
For this family, they let $\Lambda=(\alpha,a,0,0)$, which gives
$$
z_1=-\frac {aw_1}\alpha,
\ \ z_2=-\frac {aw_2}\alpha,
\ \ z_3=-\frac {a^2w_1}{\alpha^2},
\ \ z_4=-\frac {a^2w_2}{\alpha^2},\ \
\theta_1=0,\ \ \theta_2=0.
$$
From this, we get a 5-dimensional family of \textit{solvable} Lie algebras given by
\begin{eqnarray*}
\lb ZX&=&\alpha X-\frac {aw_1}\alpha Z+w_1 W,\\
\lb ZY&=&\alpha Y-\frac {aw_2}\alpha Z+w_2 W,\\
\lb WX&=&a X-\frac {a^2w_1}{\alpha^2} Z+\frac {aw_1}\alpha W,\\
\lb WY&=&a Y-\frac {a^2w_2}{\alpha^2} Z+\frac {aw_2}\alpha W.
\end{eqnarray*}

\begin{itemize}
    \item[$\mathcal{AK}:$] Since $\theta_1=\theta_2=0,$  $\alpha\neq0$ and $a\neq0,$ this family can not be almost K\"{a}hler.
    \item[$\mathcal{I}:$] From
    $$z_1=-\frac {aw_1}\alpha, \ \ z_2=-\frac {aw_2}\alpha, \ \ z_3=-\frac {a^2w_1}{\alpha^2} \ \ \text{and} \ \ z_4=-\frac {a^2w_2}{\alpha^2}$$ 
    we see that the family is integrable if and only if
$$-\frac{2aw_2}{\alpha}-\frac {a^2w_1}{\alpha^2}+w_1=0  \ \ \text{and} \ \ -\frac{2aw_1}{\alpha}+\frac {a^2w_2}{\alpha^2}-w_2=0.$$
This can be rewritten as
\begin{equation}\label{eqg17}
\left(\alpha^2- {a^2}\right)w_1=2a\alpha w_2 \ \ \text{and} \ \ \left(\alpha^2- {a^2}\right)w_2=-2a\alpha w_1.
\end{equation}
Since $2a\alpha\neq0$, we can divide and  get
\begin{equation*}
    w_2=\frac{\alpha^2-a^2}{2a\alpha}w_1
    =-\frac{(\alpha^2-a^2)^2}{4a^2\alpha^2}w_2.
\end{equation*}
Suppose that $w_2\neq0$. Then
$$4a^2\alpha^2=-(\alpha^2-a^2)^2,$$
which can be rewritten as
$$(\alpha^2+a^2)^2=0.$$
This has no nonzero, real solution. Thus, $J$ is integrable if and only if $w_1=w_2=0$. We get a $2$-dimensional family $\g_{17}^{\mathcal{I}}(\alpha,a)$ with
\begin{eqnarray*}
\lb ZX&=&\alpha X,\\
\lb ZY&=&\alpha Y,\\
\lb WX&=&a X,\\
\lb WY&=&a Y.
\end{eqnarray*}
This corresponds to a semidirect product $\R^2\ltimes\R^2$.
    \item[$\mathcal{K}:$] The family can not be K\"{a}hler since it can not be almost K\"{a}hler.
\end{itemize}
\end{example}

\begin{example}[$\g_{18}(\beta,b,z_3,z_4,\theta_1,\theta_2)$]\label{exam-F8}
Here $\Lambda=(0,0,\beta,b)$, which gives
$$z_1=\frac {\beta z_3}b,
\ \ z_2=\frac {\beta z_4}b,
\ \ w_1=-\frac {\beta^2z_3}{b^2},
\ \ w_2=-\frac {\beta^2z_4}{b^2}.
$$
We get a 6-dimensional family of Lie algebras which are \textit{not solvable} in general given by
\begin{eqnarray*}
\lb ZX&=& \beta Y+\frac {\beta z_3}b Z-\frac {\beta^2z_3}{b^2} W,\\
\lb ZY&=&-\beta X+\frac {\beta z_4}b Z-\frac {\beta^2z_4}{b^2} W,\\
\lb WX&=& b Y+z_3 Z-\frac {\beta z_3}b W,\\
\lb WY&=&-b X+z_4 Z-\frac {\beta z_4}b W,\\
\lb YX&=&\theta_1 Z+\theta_2 W.
\end{eqnarray*}

\begin{itemize}
    \item[$\mathcal{AK}:$]
    Since $a=\alpha=0,$ this family is almost K\"{a}hler if and only if $\theta_1=\theta_2=0.$
    We get a $4$-dimensional family
    $\g_{18}^\mathcal{AK}(\beta,b,z_3,z_4)$ of \textit{solvable} Lie algebras with
\begin{eqnarray*}
\lb ZX&=& \beta Y+\frac {\beta z_3}b Z-\frac {\beta^2z_3}{b^2} W,\\
\lb ZY&=&-\beta X+\frac {\beta z_4}b Z-\frac {\beta^2z_4}{b^2} W,\\
\lb WX&=& b Y+z_3 Z-\frac {\beta z_3}b W,\\
\lb WY&=&-b X+z_4 Z-\frac {\beta z_4}b W.
\end{eqnarray*}
    \item[$\mathcal{I}:$] Here, $$z_1=\frac {\beta z_3}b, \ \ z_2=\frac {\beta z_4}b, \ \ w_1=-\frac {\beta^2z_3}{b^2} \ \ \text{and} \ \ w_2=-\frac {\beta^2z_4}{b^2}$$ and we can easily see that $J$ is integrable if and only if
$$
(\beta^2-b^2)z_3=2\beta bz_4 \ \ \text{and} \ \ (\beta^2-b^2)z_4=-2\beta bz_3.
$$
We can immediately see that this has 
the same form as equation \eqref{eqg17}. Thus, $J$ is integrable if and only if $z_3=z_4=0.$
We get the $4$-dimensional family
$\g_{18}^\mathcal{I}(\beta,b,\theta_1,\theta_2)$ of Lie algebras which are \textit{not solvable} in general and have the Lie bracket relations
\begin{eqnarray*}
\lb ZX&=& \beta Y,\\
\lb ZY&=&-\beta X,\\
\lb WX&=& b Y,\\
\lb WY&=&-b X,\\
\lb YX&=&\theta_1 Z+\theta_2 W.
\end{eqnarray*}
    \item[$\mathcal{K}:$]
    The family is K\"{a}hler if and only if
    $\theta_1=\theta_2=0,$ and $z_3=z_4=0.$
    From this, we get a 2-dimensional family $\g_{18}^\mathcal{K}(\beta,b)$ of \textit{solvable} Lie algebras with
\begin{eqnarray*}
\lb ZX&=& \beta Y,\\
\lb ZY&=&-\beta X,\\
\lb WX&=& b Y,\\
\lb WY&=&-b X.
\end{eqnarray*}
We see that this corresponds to a semidirect product $\R^2\ltimes\R^2$.
\end{itemize}
\end{example}

\begin{example}[$\g_{19}(\alpha,\beta,w_1,w_2)$]\label{exam-F9}
Here $\Lambda=(\alpha,0,\beta,0)$, which gives
$$z_1=z_2=z_3=z_4=\theta_1=\theta_2=0.$$
and we get a 4-dimensional family of \textit{solvable} Lie algebras where
\begin{eqnarray*}
\lb ZX&=&\alpha X +\beta Y+w_1 W,\\
\lb ZY&=&-\beta X+\alpha Y+w_2 W.
\end{eqnarray*}
\begin{itemize}
    \item[$\mathcal{AK}:$] This family can not be almost K\"{a}hler since $a=\theta_1=\theta_2=0$ and $\alpha\neq0.$
    \item[$\mathcal{I}:$] Since $z_1=z_2=z_3=z_4=0,$ $J$ is integrable if and only if $w_1=w_2=0.$
    We get the $2$-dimensional family  $\g_{19}^\mathcal{I}(\alpha,\beta)$
    given by
\begin{eqnarray*}
\lb ZX&=&\alpha X +\beta Y,\\
\lb ZY&=&-\beta X+\alpha Y.
\end{eqnarray*}
This corresponds to a semidirect product $\R^2\ltimes\R^2$.
    \item[$\mathcal{K}:$] Since this family can not be almost K\"{a}hler, it can not be K\"{a}hler.
\end{itemize}
\end{example}

\begin{example}[$\g_{20}(\alpha,a,\beta,w_1,w_2)$]\label{exam-F10}
Here they let $\Lambda=(\alpha,a,\beta,b)$, which gives
$$
z_1=-\frac {aw_1}\alpha,
\ \ z_2=-\frac {aw_2}\alpha,
\ \ z_3=-\frac {a^2w_1}{\alpha^2},
\ \ z_4=-\frac {a^2w_2}{\alpha^2},$$
$$b=\frac{\beta a}\alpha,\ \ \theta_1=0\ \ \textrm{and} \ \ \theta_2=0.$$
We get a 5-dimensional family of \textit{solvable} Lie algebras with
\begin{eqnarray*}
\lb ZX&=&\alpha X+\beta Y-\frac {aw_1}\alpha Z+w_1 W,\\
\lb ZY&=&-\beta X+\alpha Y-\frac {aw_2}\alpha Z+w_2 W,\\
\lb WX&=&a X+\frac{\beta a}\alpha Y-\frac {a^2w_1}{\alpha^2} Z+\frac a\alpha w_1 W,\\
\lb WY&=&-\frac{\beta a}\alpha X+a Y-\frac {a^2w_2}{\alpha^2} Z+\frac a\alpha w_2 W.
\end{eqnarray*}

\begin{itemize}
    \item[$\mathcal{AK}:$] Since $\theta_1=\theta_2=0,$ $\alpha\neq0$ and $a\neq0,$ this family can not be almost K\"{a}hler.
    \item[$\mathcal{I}:$] The constants $z_1,\dots,z_4,$ $w_1$ and $w_2$ are the same as for the family $\g_{17}$. Thus $J$ is integrable if and only if $w_1=w_2=0$.
    This gives a $3$-dimensional family $\g_{20}^\mathcal{AK}(\alpha,a,\beta)$ with
\begin{eqnarray*}
\lb ZX&=&\alpha X+\beta Y,\\
\lb ZY&=&-\beta X+\alpha Y,\\
\lb WX&=&a X+\frac{\beta a}\alpha Y,\\
\lb WY&=&-\frac{\beta a}\alpha X+a Y.
\end{eqnarray*}
We see that this corresponds to a semidirect product $\R^2\ltimes\R^2$.
    \item[$\mathcal{K}:$] The family can not be K\"{a}hler since it can not be almost K\"{a}hler.
\end{itemize}
\end{example}

\phantomsection   
\addcontentsline{toc}{chapter}{Bibliography} 

\backcover

\begin{thebibliography}{99}\label{biblio}
\normalsize

\bibitem{Bai-Eel}
P.~Baird and J.~Eells,
{\it A conservation law for harmonic maps},
Geometry Symposium Utrecht 1980,
Lecture Notes in Mathematics {\bf 894}, 1-25, Springer (1981).	

\bibitem{Bai-Woo-book}
P. Baird, J.C. Wood,
{\it Harmonic Morphisms Between Riemannian Manifolds},
The London Mathematical Society Monographs {\bf 29},
Oxford University Press (2003).

\bibitem{Ballman} W. Ballmann, {\em Lectures on K\"{a}hler Manifolds}, ESI Lectures in Mathematics and Physics, European Mathematical Society (EMS) (2006).


\bibitem{Che}
K. Chen,
{\it Eight-Dimensional Hermitian Lie Groups Conformally Foliated by  Minimal $\SU{2}\times\SU{2}$ Leaves},
Bachelor's thesis, Lund University (2021), \\
\href{http://www.matematik.lu.se/matematiklu/personal/sigma/students/Kexing-Chen-BSc.pdf}{\tt www.matematik.lu.se/matematiklu/personal/sigma/students/Kexing-Chen-BSc.pdf}

\bibitem{Fug}
B. Fuglede, {\it Harmonic morphisms between Riemannian manifolds},
Ann. Inst. Fourier {\bf 28} (1978), 107-144.

\bibitem{Gra-Her}
A. Gray, L. M. Hervella,
{\it The sixteen classes of almost Hermitian manifolds and their linear invariants},
Ann. Mat. Pura Appl. {\bf 123} (1980), 35-58.

\bibitem{Gud-Sve-6}
S. Gudmundsson, M. Svensson,
{\it Harmonic morphisms from four-dimensional Lie groups},
J. Geom. Phys. {\bf 83} (2014), 1-11.

\bibitem{Gud-10}
S. Gudmundsson,
{\it Holomorphic harmonic morphisms from four-dimensional non-Einstein manifolds},
Internat. J. Math. {\bf 26} (2015), 1550006 [7 pages].

\bibitem{FiveDimLie}
S. Gudmundsson, {\em Harmonic morphisms from five-dimensional Lie groups}, Geom. Dedicata
\textbf{184} (2016), 143-157.

\bibitem{Gud-Riemann}
S. Gudmundsson,
{\it An Introduction to Riemannian Geometry},  Lecture Notes in Mathematics,
Lund University (2021),\\
\href{http://www.matematik.lu.se/matematiklu/personal/sigma/Riemann.pdf}{\tt www.matematik.lu.se/matematiklu/personal/sigma/Riemann.pdf}


\bibitem{Hsiung} C.C. Hsiung, {\em Almost Complex and Complex Structures}, Series in Pure Mathematics \textbf{20}, {World Scientific} (1995).




\bibitem{Ishihara} T. Ishihara, \textit{A mapping of Riemannian manifolds which preserves harmonic functions}, J. Math. Soc. Japan \textbf{7} (1979), 345-370.

\bibitem{Jac}
C. G. J. Jacobi,
{\it \"{U}ber eine L\"{o}sung der partiellen Differentialgleichung $\Delta(V)=0$},
J. Reine Angew. Math. {\bf 36} (1848), 113-134.

\bibitem{Knapp}
A.W. Knapp,
{\it Lie Groups Beyond an Introduction},
Progress in Mathematics {\bf 140},
Birkh\"{a}user Boston Inc., Boston (2002).

\bibitem{New-Nir} A. Newlander, L. Nirenberg, {\em Complex analytic coordinates in almost complex manifolds}, Ann.
of Math. \textbf{65} (1957), 391-404.






\end{thebibliography}
\end{document}